\documentclass[12pt,reqno]{amsart}
\usepackage{a4wide}
\numberwithin{equation}{section}

\usepackage{amsmath}
\usepackage{color}
\usepackage{txfonts}
\usepackage{amsfonts}
\usepackage{amsmath,amsthm,amssymb,amscd}
\usepackage{latexsym}
\usepackage{hyperref}
\usepackage[numbers,sort&compress]{natbib}
\usepackage{hypernat}
\allowdisplaybreaks

\numberwithin{equation}{section}

\newtheorem{theorem}{Theorem}[section]
\newtheorem{proposition}[theorem]{Proposition}

\newtheorem{lemma}[theorem]{Lemma}

\theoremstyle{definition}

\newtheorem{remark}[theorem]{Remark}
\newtheorem{remarks}[theorem]{Remarks}

\newcommand{\va}{\varepsilon}

\newcommand{\ga}{\gamma}

\newcommand{\la}{\lambda}

\newcommand{\dx}{\,\mathrm{d}x}
\newcommand{\dy}{\,\mathrm{d}y}
\newcommand{\dsi}{\,\mathrm{d}\sigma_x}
\newcommand{\dt}{\,\mathrm{d}t}

\def\r{\mathbb{R}}

\begin{document}

\title[Normalized Solutions for nonlinear Schr\"{o}dinger-Poisson equations]
{Normalized Solutions for nonlinear Schr\"{o}dinger-Poisson equations involving nearly mass-critical exponents}

\author{Qidong Guo \textsuperscript{1}, Rui He \textsuperscript{1}, Qiaoqiao Hua \textsuperscript{1,$\ast$}, Qingfang Wang \textsuperscript{2} 
}

\address{ \textsuperscript{1} School  of Mathematics and Statistics, Central  China Normal University, Wuhan 430079, P. R. China }
\address{ \textsuperscript{2} School of Mathematics and Computer Science, Wuhan Polytechnic University, Wuhan 430023, P. R. China}

\email{qdguo@mails.ccnu.edu.cn}

\email{hry@mails.ccnu.edu.cn}

\email{hqq@mails.ccnu.edu.cn}

\email{hbwangqingfang@163.com}

\thanks{$^\ast$ Corresponding author: Qiaoqiao Hua}

\begin{abstract}
We study the Schr\"{o}dinger-Poisson-Slater equation
\begin{equation*}\left\{\begin{array}{lll}
		-\Delta u + \la u +  \big(|x|^{-1} \ast |u|^{2}\big)u = V(x) u^{ p_\va-1 } , \,  \text{ in }  \r^{3},\\ [2mm]
		\int_{\r^3}u^2 \dx= a,\,\, u > 0, u \in H^{1}(\mathbb{R}^{3}),
	\end{array}
	\right.
\end{equation*}
where $\lambda$ is a Lagrange multiplier, $V(x)$ is a real-valued potential, $a\in \r_{+}$ is a constant, $ p_{\va} = \frac{10}{3} \pm \va$ and $\va>0$ is a small parameter.
In this paper, we prove that it is the positive critical value of the potential $V$
 that affects the existence of single-peak solutions for this problem.
Furthermore, we prove the local uniqueness of the solutions we construct.


\end{abstract}

\maketitle

\section{Introduction}

We consider the following Schrödinger-Poisson-Slater equation
\begin{equation}\label{lambda-12}\left\{\begin{array}{lll}
		-\Delta u + \lambda  u - \gamma \big(|x|^{-1} \ast |u|^{2}\big)u= V(x) |u|^{p-2 } u ,    \text{ in }  \r^{3},\\[2mm]		
		\int_{\r^3}u^2 \dx= a, \,\,u \in H^{1}(\mathbb{R}^{3}),
	\end{array}
	\right.
\end{equation}
where $\lambda$ is the Lagrange multiplier, $\gamma\in \r$, $p\in(2,6]$, $V(x)$ is a real-valued potential and $a>0$ is some constant. 
The minimization problem corresponding to \eqref{lambda-12} is given by
\begin{equation}\label{Schrodinger-2}
	m(a)=\inf_{u \in H^{1}(\mathbb{R}^{3})}\,\Bigl\{ F(u):\, \|u\|_{L^2(\r^3)}^2=a\Bigr\},
\end{equation}
where
\begin{equation}\label{F}
	\begin{aligned}
		F(u)={}&\frac{1}{2}\int_{\r^{3}} |\nabla u|^{2}\dx
		-\frac{\gamma}{2}\int_{\r^3}\Big(\int_{\r^3}\frac{u^2(y)}{|x-y|}\dy\Big) u^2(x)\dx
		-\frac{1}{p}\int_{\r^{3}}V(x)|u|^{p}\dx. \\
	\end{aligned}
\end{equation}

We first look at the case that $V(x)$ is a constant function. We assume $V(x)\equiv c$. The case where $\gamma<0$ and $c > 0$ in \eqref{lambda-12} has been the most studied so far. For $p \in (2,\frac{10}{3})$, the functional $F (u)$ is bounded from below
and coercive on $\Big\{u: \|u\|_{L^2(\r^3)}^2=a\Big\}$. It also holds when $p=\frac{10}{3}$ and $a > 0$ is small enough. 
Due to the non-local term, it is not easy to find a minimizer of the energy functional \eqref{F} even if $m(a)>-\infty$.
Using techniques introduced in \cite{Catto1992}, it was proved in \cite{Sanchez2004} that minimizers exist
for $p = \frac{8}{3}$ if $a > 0$ is small enough. 
Later, J. Bellazzini and G. Siciliano  in \cite{Bellazzini2011} showed the existence of  minimizers for any $p\in (2, 3)$ provided that $a > 0$ is sufficiently small, while the case $p \in (3,\frac{10}{3})$ was proved in \cite{Bellazzini2011-1} if $a>0$ is sufficiently large, see also \cite{Jeanjean2013} where the existence of a threshold value of mass $a_0 > 0$ is given. It was also presented in \cite{Jeanjean2013} that a minimizer does not exist for any $a > 0$ if $p= 3$ or $p=\frac{10}{3}$.  For $p \in (\frac{10}{3}, 6)$, a scaling argument 
reveals that $m(a) = -\infty$. However, by a mountain-pass argument, it was proved in \cite{Bellazzini2013} that there exists a critical point of $F(u)$ constrained to $\Big\{u: \|u\|_{L^2(\r^3)}^2=a\Big\}$
at a strictly positive level for $a > 0$ small enough. 
Recently, L. Jeanjean and T.~T. Le \cite{Jeanjean2021} studied \eqref{lambda-12} with $p\in(\frac{10}{3},6]$. For $\gamma>0$ and $c>0$, they studied the multiplicity of normalized solutions to \eqref{lambda-12}. Specifically, they proved that, both in the Sobolev subcritical case $p\in(\frac{10}{3},6)$ and in the Sobolev critical case $p=6$,  the problem \eqref{lambda-12} admits two solutions $u_a^+$ and $u_a^-$ if $a$ is small.
For $\gamma>0$ and $c<0$, the problem \eqref{lambda-12} admits a solution which is a global minimizer.
For $\gamma<0$ and $c>0$ and $p=6$, the problem \eqref{lambda-12} does not admit positive solutions, which complements the result of  \cite{Bellazzini2013}.  The above research results are all based on the hypothesis of trivial potentials.  Our goal is to study the relation between the existence of solutions and the nontrivial potentials.

We are also concerned about the concentration phenomenon of solutions to \eqref{lambda-12}. There are many works on blow-up results for solutions of Schr\"{o}dinger-Poisson problems without $L^2$-constraints. Most of them study the singularly perturbed problem with the small
parameter $\varepsilon>0$
\begin{equation}\label{lambda-14}\left\{\begin{array}{lll}
		-\varepsilon^2\Delta u +  V(x)u + \Phi(x) u= |u|^{p-2 } u ,   & \text{ in }  \r^{3},\\[2mm]		
		-\Delta\Phi = u^2, & \text{ in }  \r^{3}.\\
	\end{array}
	\right.
\end{equation}
In \cite{DAprile2006}, T. D'Aprile and J. Wei studied the above problem in the unit ball $B_1$ of $\r^3$ with Dirichlet boundary conditions, and they proved the existence of positive radial solutions $(u_\va, \Phi_\va)$ such that $u_\va$ concentrates at a distance $(\va/2)|\log \va|$ away from the boundary $\partial B_1$ as $\va$ tends to zero. A. Ambrosetti \cite{Ambrosetti2008} showed the existence of spike-like solutions of \eqref{lambda-14}. D. Ruiz and G. Vaira \cite{Ruiz2011} proved the existence of cluster solutions of  \eqref{lambda-14}, whose bumps concentrate around a local strict minimum of the potential $V$.  Furthermore, I. Ianni and G. Vaira \cite{Ianni2008} studied the existence of semiclassical states for a nonlinear Schr\"odinger-Poisson system that concentrates near critical points of the external potential $V(x)$. Under the suitable conditions for $V(x)$,  the semiclassical states of \eqref{lambda-14} concentrating on spheres were obtained in \cite{DAprile2005,Ruiz2005}.
The reduction method was also used for the following problem with the small
parameter $\varepsilon>0$
\begin{equation}\label{lambda-13}\left\{\begin{array}{lll}
		-\varepsilon^2\Delta u +  V(x)u + K(x) \Phi(x) u=  |u|^{p-2 } u ,   & \text{ in }  \r^{3},\\[2mm]		
		-\varepsilon\Delta\Phi = K(x)u^2, & \text{ in }  \r^{3},\\
	\end{array}
	\right.
\end{equation}
which is quite different from \eqref{lambda-14}. Here the potential $K(x)$ plays a role when the critical point of $V(x)$ is degenerate (see \cite{Ambrosetti2008} for example). In \cite{Ianni2009}, the authors found necessary conditions for solutions concentrating on a sphere with $V(x)$ and $K(x)$ being radial, and I. Ianni  \cite{Ianni2009-2} proved the existence of such solutions as $\va\to0$.

As far as we know, there is no results on normalized and concentrated solutions to Schr\"{o}dinger-Poisson systems with nontrivial potentials. 
We are concerned with the situation where the exponent $p$ approaches to $\frac{10}{3}$ for a given mass. Since the energy functional $F(u)$ with $V(x)\equiv c$ restricted to the norm constraint has no critical points for $\gamma, c<0$ (see \cite[Theorem 1.1]{Jeanjean2021}), we can at least consider the case $\gamma\in\r$ and $V(x)>0$. Specially, we study the following equation
\begin{equation}\label{Schrodinger}\left\{\begin{array}{lll}
		-\Delta u + \la  u +  \big(|x|^{-1} \ast |u|^{2}\big)u = V(x) u^{ p-1 } , \,   \text{ in } \r^{3},\\[2mm]		
		\int_{\r^3}u^2 \dx = a,\,\, u > 0, u \in H^{1}(\mathbb{R}^{3}),
	\end{array}
	\right.
\end{equation}
where $\lambda$ is the unknown Lagrange multiplier, $V(x)$ is a real-valued potential which has at least one positive critical value, $a\in \r_{+}$ is a constant, and $ p$ is close to $\frac{10}{3}$. 

Motivated by the above rich literature, we expect that the normalized solution of \eqref{Schrodinger} has a concentration phenomenon for $p\nearrow \frac{10}{3}$ and $a>0$ large enough, or for $p\searrow \frac{10}{3}$ and $a>0$ small enough. In fact, since Theorem \ref{lemma1} proved later indicates that the limit equation of \eqref{Schrodinger} is that $-\Delta u + u = u^{p-1}, \, u \in H^1(\r^3)$, we can firstly consider the case where the Poisson term vanishes. Then \eqref{Schrodinger} reduces to the Schr\"odinger equation
\begin{equation}\label{Sch1}
	\left\{\begin{array}{lll}
		-\Delta u + \la  u = V(x) u^{ p-1 } , \,  \text{ in } \r^{3},\\[2mm]		
		\int_{\r^3}u^2 \dx = a,\,\, u > 0, u \in H^{1}(\mathbb{R}^{3}).
	\end{array}
	\right.
\end{equation}
To raise our problem more clearly, we assume that $V\equiv1$.
Denote by $w\in H^1(\r^3)$ the unique radial positive solution of $-\Delta w+ w=  w^{p-1} \text{ in } \r^3$ and let $u(x)=\la^{\frac1{p-2}}w(\sqrt \lambda x)$. Then
$-\Delta u+ \lambda u=  u^{p-1}$ and
\[
a=\int_{\r^3} u^2 \dx= \lambda^{\frac2{p-2}-\frac 32}
\int_{\r^3} w^2 \dx.
\]
We are capable of solving this equation for the unknown $\la$ if $p\ne \frac {10}{3}$. Hence we see that if $p\ne \frac {10}{3}$, we
obtain a solution to \eqref{Sch1} with $V\equiv1$. Notice that if $a< \int_{\r^3} w^2\dx$ (or $a> \int_{\r^3} w^2\dx$) then $\lambda\to +\infty$ as $p\searrow \frac{10}{3} $ (or $\lambda\to +\infty$ as $p\nearrow \frac{10}{3} $).  
It is well known that when $\lambda>0$ is sufficiently large, \eqref{Sch1} without $L^2$-constraints has a concentrated solution (see \cite{del1998,Oh1988} for example), which indicates that the solution to \eqref{Sch1}  has a concentration phenomenon as $p$ tends to the mass-critical exponent. So it is natural to regard the exponent $p$ as a perturbation parameter to study the concentration properties of the solutions to \eqref{Schrodinger}.

In the rest of this paper, we study the following problem
\begin{equation}\label{maineq}\left\{\begin{array}{lll}
		-\Delta u + \lambda  u + \big(|x|^{-1} \ast |u|^{2}\big)u=  V(x) u^{p_{\varepsilon}-1 } ,  \text{ in }  \r^{3},\\[2mm]		
		\int_{\r^3}u^2  \dx = a,\,\, u > 0, \, u \in H^{1}(\mathbb{R}^{3}),
	\end{array}
	\right.
\end{equation}
where $p_{\va} = \frac{10}{3} \pm \va$ and $\va>0$ is a small parameter. We postulate the assumptions on $V(x)$:

\vskip2mm
(V).  $V\in C^2(\r^3)$,  $V(x)\ge C_0>-\infty,$  $V(x)$ has a non-degenerate critical point $b_0\in \r^3$ and $V(b_0)=V_0>0$. 
\vskip2mm

The aim of the paper is to study the existence and local uniqueness of
solutions $u_\va$ for \eqref{maineq}, satisfying
\begin{equation}\label{1-17-5}
	\max_{x \in B_{\theta}(b_0)} u_\va(x) \rightarrow +\infty, \, \text{while}\,\, u_\va(x)\rightarrow 0 \,\,\text{uniformly in}\,\, \r^3\setminus B_{\theta}(b_0),
\end{equation}
as $\va\to 0$, where $\theta>0$ is any small fixed
constant and $b_0$ is a point in $\r^3.$

Let $\bar p = \frac{10}{3}$ and denote by $Q_{\bar{p}}$ the unique radial positive solution of $-\Delta u + u = u^{\bar{p}-1}, \, u \in H^1(\r^3)$. Define
$$a_{*} := \int_{\r^3}Q_{\bar{p}}^2 \dx.$$
Then the existence result is stated  as follows.

\begin{theorem}\label{ExistenceOfNormalized} Suppose (V) holds. If one of the following cases holds:
	\begin{itemize}
		\item [(i)] $a< V_0^{-\frac32}a_{*}$ and $p_{\va} = \bar{p}+\va$;\\
		\item [(ii)]  $a> V_0^{-\frac32}a_{*}$ and $p_{\va} = \bar{p}-\va$,
	\end{itemize}
	then there exists an $\va_0>0$ such that for $\va\in(0,\va_0)$,  \eqref{maineq} has a  solution $u_\va$ satisfying \eqref{1-17-5}.
	Moreover, we have $\la \in \Big(e^{\frac{4}{9\va} \ln (V_0^{-\frac32}\frac{a_{*}}{a})}, e^{\frac{16}{9\va} \ln (V_0^{-\frac32}\frac{a_{*}}{a})}\Big)$  for case $(i)$  and $\la \in \Big(e^{\frac{4}{9\va} \ln (V_0^{\frac32}\frac{a}{a_{*}})}, e^{\frac{16}{9\va} \ln (V_0^{\frac32}\frac{a}{a_{*}})}\Big)$  for case $(ii)$.
\end{theorem}
\begin{remarks}\label{remarks}
	(1). The case (i) and (ii) of Theorem \ref{ExistenceOfNormalized} extend some of the results in \cite{Jeanjean2021} and \cite{Bellazzini2011-1}  to the case with potential functions, respectively. And Theorem \ref{ExistenceOfNormalized} provides the effect of potential functions on the existence of solutions. Unlike \cite{Jeanjean2021,Bellazzini2011-1}, we adopt a finite dimensional reduction method to prove the existence of solutions. Due to the effect of Poisson terms, we have to put the potential $V(x)$ into the nonlinear term $u^{p_\va-1}$ to ensure that the approximate solution is ``good''. We will give a profound analysis near the equation \eqref{lambda-2}.
	
	(2). It is the positive critical value of the potential $V$
	that affects the existence of single-peak solutions for problem \eqref{maineq}. Precisely, the positivity of the critical value $V_0$  ensures the existence of positive solutions, as well as that the manifold $\Big\{u: \|u\|_{L^2(\r^3)}=a\Big\}$ is not empty. 
	
	(3). We are also concerned about the normalized multi-peak solutions, but the Poisson term will lead to delicate interactions among peaks.
	Similar problems also appear in \cite{Luo2020,Li2021} where the authors pointed out that the asymptotic behaviors of concentration points are quite different from those of Schr\"odinger equations due to the nonlocal terms.  We postpone the study to future work.
	
	(4). We also remark that one can consider the problem \eqref{Schrodinger} by regarding the mass $a$ as a perturbed parameter. This strategy can be found in \cite{Pellacci2021,Yan2021,Guo2}. The solutions to \eqref{Schrodinger} with $p=\bar{p}$ will also exhibit a concentration phenomenon as $a\to V_0^{-\frac32}a_{*}$. Precisely, by careful calculations of Pohozaev identities (similar to  \cite[Proposition 3.5]{Yan2021}), one can show that the solution $u_a$ to \eqref{Schrodinger} concentrates at a non-degenerate critical point of $V(x)$ when $a\nearrow V_0^{-\frac32}a_{*}$ and $\Delta V(b_0)>0$, or $a\searrow V_0^{-\frac32}a_{*}$ and $\Delta V(b_0)<0$.
\end{remarks}

Now, let us outline the proof of Theorem \ref{ExistenceOfNormalized}.

Firstly, we consider \eqref{maineq} without $L^2$-constraints, namely
\begin{equation}\label{lambda-1}\left\{\begin{array}{lll}
-\Delta u + \lambda  u +  \big(|x|^{-1} \ast |u|^{2}\big)u= V(x) u^{p_{\varepsilon}-1 } ,    \text{ in }  \r^{3},\\[2mm]
u > 0, u \in H^{1}(\mathbb{R}^{3}).
\end{array}
\right.
\end{equation}
Let us regard $\lambda$ as the perturbation parameter for a moment. Then for any fixed $\va>0$ small, we can use Lyapunov-Schmidt reduction method to obtain a concentrated solution. We refer to the proof in \cite{Ambrosetti2008,DAprile1,Yu2021} for example. Precisely, \eqref{lambda-1}
has  a solution $u_\lambda$ of the form
\begin{equation}\label{2-1}
u_{\la} = U_{x_{\la},p_\va} + \omega_\la,
\end{equation}
where  $U_{x_{\la},p_\va}(x)=\big(\frac{\la}{V_0}\big)^{\frac{1}{p_{\va}-2}}Q_{p_\va} \big(\sqrt{\lambda} (x - x_{\la})\big)$,
$V_0=V(b_0)>0$,
$Q_{p_\va}(x)$ is  the unique radial solution  to
\begin{equation}\label{1-31-5}
 - \Delta Q + Q = Q^{p_{\va}-1}, \,Q>0,\, Q \in H^{1}( \mathbb{R}^{3}),
\end{equation}
$x_{\la}\to b_0$, and $\int_{\mathbb{R}^{3}} \bigl(|\nabla \omega_\la|^{2} +(\la+V(x)) \omega_\la^{2} \bigr)  \dx=o\bigl(\lambda^{\frac{2}{p_\va-2}-\frac 12}\bigr)$  as $\lambda\to +\infty$.

Let us make a further elaboration about Remark \ref{remarks} (1) here. The problem \eqref{lambda-1} is quite different from \eqref{lambda-14} and \eqref{lambda-13}. If we consider the following singularly perturbed problem 
\begin{equation}\label{lambda-2}\left\{\begin{array}{lll}
		-\Delta u + \lambda  u + V(x) u + \big(|x|^{-1} \ast |u|^{2}\big)u=  u^{p-1 } ,    \text{ in }  \r^{3},\\[2mm]
		u > 0, u \in H^{1}(\mathbb{R}^{3}),
	\end{array}
	\right.
\end{equation}
for fixed $p\in(2,6)$, then, by revisiting calculations in Lemma \ref{2-2}, we will find that the Poisson term $ \big(|x|^{-1} \ast |u|^{2}\big)u$ plays a dominant role compared to the potential term $V(x) u$, which leads to the negligible influence of the potential $V(x)$ on the location of concentration points. When we consider the problem \eqref{lambda-1},  Lemma \ref{2-2} shows that the contributions of $V(x) u^{p_{\varepsilon}-1 }$ and $\big(|x|^{-1} \ast |u|^{2}\big)u$ to the error term $\omega_{\la}$ are roughly the same which is very bad to us. However, thanks to the symmetry property \eqref{symetry}, we can still obtain the effect of the potential $V(x)$ on the location of concentration points in solving a finite dimensional problem.  In view of the above analysis, it is also natural to consider the following equation
\begin{equation}\label{Schrodinger3}\left\{\begin{array}{lll}
		-\Delta u + \la  u +  \big(|x|^{-1} \ast V(x)|u|^{2}\big)V(x)u = u^{ p-1 } , \,   \text{ in } \r^{3},\\[2mm]		
		\int_{\r^3}u^2 \dx = a,\,\, u > 0, u \in H^{1}(\mathbb{R}^{3}),
	\end{array}
	\right.
\end{equation}
and we believe that it is also the critical points of $V(x)$ that affect the existence of concentrated solutions to \eqref{Schrodinger3}.

Secondly, we will focus on determining the dependence of $\la$ on $\va$.
Let
\[\bar u_{\la} = \frac{\sqrt{a} u_{\la}}{\|u_{\la}\|_{L^{2}(\mathbb{R}^{3})}},\]
then $\bar u_{\la}$ satisfies the equation
\begin{equation}\label{15-18-5}\left\{\begin{array}{lll}
-\Delta \bar u_{\la} + \lambda \bar u_{\la} 
+  \Bigl(\frac{ \| u_{\la}\|_{L^{2}(\r^3)}}{\sqrt a}\Bigr)^2 (|x|^{-1} \ast |\bar u_{\la}|^{2})\bar u_{\la}
= V(x) \Bigl(\frac{ \| u_{\la}\|_{L^{2}(\r^3)}}{\sqrt a}\Bigr)^{p_{\va}-2}{\bar u}_{\la}^{p_{\va} -1} ,   \text{ in } \r^{3},\\[2mm]
\int_{\r^3}|\bar u_{\la}|^{2}dx =a. 
\end{array}
\right.
\end{equation}
We will obtain a solution to \eqref{maineq} once we prove that $\frac{\| u_{\la}\|_{L^{2}(\r^{3})}}{\sqrt{a}} =1$.
In view of Theorem~\ref{ExistenceOfNormalized}, we expect 
that for each $\va>0$ small, there is $\la \in \Big(e^{\frac{4}{9\va} \ln (V_0^{-\frac32}\frac{a_{*}}{a})}, e^{\frac{16}{9\va} \ln (V_0^{-\frac32}\frac{a_{*}}{a})}\Big)$  for case $(i)$  $\Big($or $\la \in \Big(e^{\frac{4}{9\va} \ln (V_0^{\frac32}\frac{a}{a_{*}})}, e^{\frac{16}{9\va} \ln (V_0^{\frac32}\frac{a}{a_{*}})}\Big)$  for case $(ii)$ $\Big)$,
  such that $\frac{\| u_{\la}\|_{L^{2}(\r^{3})}}{\sqrt{a}} =1$.

To carry out the above procedure, there are a couple of issues we need to deal with.

The first one is the existence of solutions of the form
\eqref{2-1} for \eqref{lambda-1} if $\lambda >e^{\frac{4}{9\va} \ln (V_0^{-\frac32}\frac{a_{*}}{a})}$  for case $(i)$  $\Big($or $\la > e^{\frac{4}{9\va} \ln (V_0^{\frac32}\frac{a}{a_{*}})}$  for case $(ii)$ $\Big)$. All the above results on the existence and local uniqueness of peak
solutions for \eqref{lambda-1} are obtained with the assumption that $p_\va$ is fixed
and no estimate on the parameter $\lambda$ is given.  To solve this issue, our idea is to first give a existence result of solutions of the form
\eqref{2-1} for \eqref{lambda-1} for $\la>\la_0$, where $\la_0$ is independent of $\va$, and then establish the relation between $\la$ and $\va$.

The second one is that to solve $\frac{\| u_{\la}\|_{L^{2}(\r^{3})}}{\sqrt{a}}=1$, we need to
prove the continuity of $\| u_{\la}\|_{L^{2}(\r^3)}$ as a function of $\lambda$. Hence, it is essential to prove the uniqueness
of solutions of the  form \eqref{2-1}, we refer to \cite{Guo2,Li2021,Luo2023} for example. We also need to  revisit the problem for the local uniqueness of peak
solutions of \eqref{lambda-1} to make sure that such results can be obtained for
all $\lambda>\lambda_0$ with $\lambda_0>0$ independent of $\va>0$. 

Thanks to the uniformity of $\la$ and uniqueness of peak
solutions, we can explore the dependence of $\la$ on $\va$ through the mass constraint. This indicates that the perturbation parameter can be changed from $\la$ to $\va$, and then Theorem \ref{ExistenceOfNormalized} is proved. 
To address all these issues, it requires us to estimate
$\| Q_{p_\va}-Q_{\bar{p}}\|_{H^{1}(\mathbb{R}^{3})}$ (See Lemma \ref{4-1} in Appendix A).

In the proof of Theorem~\ref{ExistenceOfNormalized}, we have constructed a single-peak solution, such that it has exactly
one peak in $B_{\theta}(b_0)$. In this paper, we
 will also prove the local uniqueness of such solution.

\begin{theorem}\label{UniquenessOfNormalized} Suppose (V) holds. Let $u_\varepsilon^{(1)}$ and $u_\varepsilon^{(2)}$ be two single-peak solutions of \eqref{maineq},
 which satisfy \eqref{1-17-5}. Then there exists a small positive number $\varepsilon_0$,  such that $u_\varepsilon^{(1)} = u_\varepsilon^{(2)}$ for all $\varepsilon$ with $0<\varepsilon < \varepsilon_0$.
\end{theorem}

There is a huge literature on local uniqueness results for spike-like solutions. We refer to \cite{Glangetas,Grossi,Cao3,Deng,Guo5,Guo4,Yan2021,Li2021} and the references therein.
What we need to point out is that the solution of \eqref{maineq} is a pair $(\la_{\va}, u_\va)$.  Thus we need to clarify the dependence of the Lagrange multiplier $\lambda_\va$ on $\va$ more precisely (See Lemma \ref{prop-lambda}). The unknown $\la$ in problem \eqref{maineq} plays a similar role to the altitude of bubbles in  the semi-linear elliptic problems with critical Sobolev exponent, see  \cite{Deng}. Inspired by it, we adopt a similar strategy to prove Theorem~\ref{UniquenessOfNormalized} in this paper. Moreover, we shall deal with the nonlocal terms in local Pohozaev identities, which forces us to consider single-peak solutions only.

\medskip

This paper is organized as follows. In Section 2, we revisit the
singularly perturbed problem \eqref{lambda-1} and prove
the  existence and local uniqueness of solutions for $\lambda>\lambda_0$, where
$\lambda_0>0$ is a large constant which is independent of $\va>0$.
 We show the existence and local uniqueness of normalized solutions to \eqref{maineq} in Section 3. In the Appendix, we list a useful estimate of the solution to \eqref{1-31-5} and the well known Hardy-Littlewood-Sobolev inequality which are frequently used in this paper.
 
 \medskip
 
 \textbf{Notations.} The symbol $C$ represents a positive constant that may change from place to place but is independent of $\la$ and $\va$.  The symbol $O(t)$ means $|O(t)|/ |t|\leq C$.
The symbol $o_\la(1)$ means some infinitesimal which tends to zero as $\la\to +\infty$,  the symbol $o_\va(1)$ means some infinitesimal which tends to zero as $\va\to 0$, and the symbol $o(1)$ means some infinitesimal which uniformly tends to zero for $\va\in(0,\va_0)$ as $\la\to +\infty$. Moreover, the symbols $o_\la(t)$, $o_\va(t)$ and $o(t)$ mean $o_\la(t)/t =o_\la(1)$, $o_\va(t)/t =o_\va(1)$ and $o(t)/t =o(1)$ respectively.

\section{Revisit the singularly perturbed problem}

In this section, we revisit the following singularly perturbed problem
\begin{equation}\label{lambda}\left\{\begin{array}{lll}
-\Delta u + \lambda u +  \big(|x|^{-1} \ast |u|^{2}\big)u= V(x) u^{p_{\varepsilon}-1 } ,   \text{ in }  \r^{3},\\[2mm]
u > 0, u \in H^{1}(\mathbb{R}^{3}).
\end{array}
\right.
\end{equation}
We assume that the function  $V(x)$ satisfies the condition (V).

\subsection{Existence}

We aim to
 construct a positive solution of  the form \eqref{2-1}, namely
 \[u_{\la} = U_{x_{\la},p_\va} + \omega_\la,\]
 where the approximate solution is defined as
 \[U_{x_{\la},p_\va}(x)=\Big(\frac{\la}{V_0}\Big)^{\frac{1}{p_{\va}-2}}Q_{p_\va} \big(\sqrt{\lambda} (x - x_{\la})\big),\]
 where $V_0=V(b_0)>0$, $Q_{p_\va}(x)$ is the solution of \eqref{1-31-5}, and $\omega_\la$ is an error function.

Define a norm in $H^{1}(\mathbb{R}^{3})$ by
\[ \|u\|_{\la} :=  \Bigl(\int_{\mathbb{R}^{3}} \bigl(|\nabla u|^{2} +\la u^{2} \bigr) \dx \Bigr)^{\frac{1}{2}}, \]
endowed with the inner product
\[ \langle u, v\rangle_{\la} =\int_{\mathbb{R}^{3}} \bigl(\nabla u \nabla v+\la u v \bigr) \dx,  \, \, u, v \in H^{1}(\mathbb{R}^{3}) .\]
Denote the function space $H_\la$ by
\[
H_\la:=\Big\{u\in H^{1}(\mathbb{R}^{3}):  \|u\|_{\la}<+\infty\Big\}.
\]
Let $\theta>0$ be a small constant.  For any $x_{\la}\in B_\theta(b_0)$, we define
\begin{equation}\label{2-e}
E_{\lambda ,p_\va} := \biggl\{u \in H^{1}(\mathbb{R}^{3}): \langle u, \frac{\partial U_{x_{\la}, p_{\va}}}{\partial x_{j}}\rangle_{\la}=0, j=1,2,3 \biggr\}.
\end{equation}

It is well known that for any given $p_\va\in(2,6)$, there exists a constant $c_\va>0$,
depending on $\va$, such that for any $\lambda>c_\va$, \eqref{lambda}
has a solution of the form \eqref{2-1}, satisfying
$x_{\la}\to b_0$, $\omega_\la\in E_{\lambda ,p_\va} $, and $\|\omega_\la\|_{\la}^{2} =o_\la\bigl(\lambda^{\frac{2}{p_\va-2}-\frac 12}\bigr)$  as $\lambda\to +\infty$. But these may not hold if $\va$ tends to $0$.
In order to prove Theorem~\ref{ExistenceOfNormalized}, we need to
find a uniform bound for $c_\va>0$. With the
help of Lemma~\ref{4-1}, we will prove the following result.

\begin{theorem}\label{ExistenceWithoutNormalized} Suppose (V) holds. Then
there exist constants $\va_0>0$ and  $\lambda_0>0$, such that for
any $\lambda>\lambda_0$ and $0<\va<\va_0$,
 \eqref{lambda} has a solution of the form
\begin{equation}\label{2}
u_{\la} =  U_{x_{\la},p_\va} + \omega_\la,
\end{equation}
where  $\omega_\la\in E_{\lambda ,p_\va} $. Moreover, $x_{\la}$ satisfies $|x_{\la}-b_0| = O(\la^{-\frac{1}{2}})$ and
\[\|\omega_{\la}\|_{\la}\leq C \Bigl(  \big| \nabla V(x_{\la}) \big| \la^{\frac{1}{p_{\va}-2} -\frac {3}{4}}  +  \la^{\frac{3}{p_{\va}-2} -\frac{9}{4}}\Bigr). \]
\end{theorem}

Suppose equation \eqref{lambda} has a solution of the form \eqref{2}, then $\omega_{\la}$ satisfies the following problem
\begin{equation}\label{2-1-0}
L_{\la} \omega = l_{\la} + R_{\la}(\omega), \,\, \omega \in E_{\lambda ,p_\va},
\end{equation}
where $l_\lambda\in E_{\lambda ,p_\va}$ is determined by
\[
\begin{aligned}
	\langle l_{\la}, \varphi\rangle_{\la} =& 
	\int_{\mathbb{R}^3}\big[V(x)- V_0\big]U_{x_{\la},p_\va}^{p_\va -1}\varphi \dx 
	- \int_{\r^3}\big(|x|^{-1}\ast U_{x_{\la},p_\va} ^2\big)U_{x_{\la},p_\va} \varphi \dx
	,\quad\forall \varphi \in E_{\lambda ,p_\va}, 
\end{aligned}
\]
the linear map $L_\la:  E_{\lambda ,p_\va}\to E_{\lambda ,p_\va}$ is determined
by
\[
\begin{aligned}
	\langle L_{\la} \omega , \varphi \rangle_{\la} =& \int_{\mathbb{R}^{3}} \Bigl[\nabla \omega \nabla \varphi +  \lambda \omega \varphi   
	-(p_{\va} -1)V(x) U_{x_{\la},p_\va} ^{p_{\va}-2} \omega \varphi  \\
	&+ \big(|x|^{-1}\ast U_{x_{\la},p_\va} ^2\big) \omega \varphi
	+2 \big(|x|^{-1}\ast (U_{x_{\la},p_\va} \omega)\big)  U_{x_{\la},p_\va}  \varphi
	\Bigr] \dx ,  \quad \forall \varphi \in E_{\lambda ,p_\va}, 
\end{aligned} 
\]
and $R_{\la}(\omega) \in E_{\lambda ,p_\va}$ is determined by
\[ 
\begin{aligned}
	\langle R_{\la}(\omega), \varphi \rangle_{\la} =& \int_{\mathbb{R}^{3}} V(x)\Bigl[ (U_{x_{\la},p_\va}  +\omega)^{p_{\va}-1}  - U_{x_{\la},p_\va} ^{p_{\va}-1} - (p_{\va} -1) U_{x_{\la},p_\va} ^{p_{\va}-2} \omega\Bigr] \varphi \dx \\
	&-2 \int_{\r^3}\big(|x|^{-1}\ast (U_{x_{\la},p_\va} \omega)\big) \omega \varphi \dx
	- \int_{\r^3}\big(|x|^{-1}\ast \omega^2\big) \omega \varphi \dx
	 , \quad \forall \varphi \in E_{\lambda ,p_\va}.
\end{aligned}
\]

The procedure to prove Theorem~\ref{ExistenceWithoutNormalized} consists of two parts. one part is to prove the existence of $\omega_{\la}\in E_{\lambda ,p_\va}$ such that \eqref{2-1-0} holds, so that the problem \eqref{lambda} is reduced to a finite dimensional problem. Another part is to prove the existence of $x_\la$ such that the finite dimensional problem is solvable.
We will focus on the estimates which are uniform for $\va\in (0, \va_0)$ in the proof of Theorem~\ref{ExistenceWithoutNormalized}.

\begin{lemma}\label{lemma1}
	There exists a constant $C>0$, independent of $\va$, such that 
	\[
	\|L_\la u\|_\la \leq C \|u\|_\la, \quad \forall u \in H_\la.
	\]
\end{lemma}
\begin{proof}
	It is obvious that $\langle L_{\la}u, v \rangle_{\la}$ is bi-linear for any $u,v\in H_\la$. Then it is sufficient to prove $L_\la$ is bounded.
	For any $\varphi\in H_{\la}$, 
	\begin{equation}\label{phi}\begin{split}
			\|\varphi\|_{L^{q}(\mathbb{R}^{3})} = & \la^{-\frac{3}{2q}}\biggl( \int_{\mathbb{R}^{3}} \Big|\varphi\Big(\frac{x}{\sqrt{\la}}\Big) \Big|^{q} \dx \biggr)^{\frac{1}{q}} \\
			\leq & C  \la^{-\frac{3}{2q}}\biggl(   \int_{\mathbb{R}^{3}} \bigg[ \Big|\nabla_{x}\Big[ \varphi\Big(\frac{x}{\sqrt{\la}}\Big)\Big] \Big|^{2} + \varphi^{2}\Big(\frac{x}{\sqrt{\la}}\Big)\bigg] \dx \biggr)^{\frac{1}{2}}\\
			\leq & C\la^{-\frac{3}{2q}+\frac14} \|\varphi\|_{\la}.
		\end{split}
	\end{equation}
 	Let $\langle L_{1,\la}u,v\rangle_\la=\int_{\mathbb{R}^3}(\nabla u\nabla v+\la uv)\dx-(p_\va-1)\int_{\mathbb{R}^3}V(x)U_{x_{\la},p_\va} ^{p_\va-2}uv\dx$ and $\langle L_{2,\la}u,v\rangle_\la=\langle L_{\la}u,v\rangle_\la-\langle L_{1,\la}u,v\rangle_\la.$
	For any $u,v\in H_{\la}$, by \eqref{phi} we have
	\begin{equation}\nonumber
		\begin{aligned}
			|\langle L_{1,\la}u,v\rangle_\la|
			&\leq|\langle u,v\rangle_{\la}|+C\int_{\mathbb{R}^3}U_{x_{\la},p_\va} ^{p_\va-2}|u||v|\dx\\
			&\leq||u||_{\la}||v||_{\la}+C\Big(\int_{\mathbb{R}^3}U_{x_{\la},p_\va} ^{p_\va}\dx\Big)^{\frac{p_\va-2}{p_\va}}
			\Big(\int_{\mathbb{R}^3}|u|^{p_\va}\dx\Big)^{\frac{1}{p_\va}}\Big(\int_{\mathbb{R}^3}|v|^{p_\va}\dx\Big)^{\frac{1}{p_\va}}\\
			&\leq C||u||_{\la}||v||_{\la}.
		\end{aligned}
	\end{equation}
	To estimate $|\langle L_{2,\la}u,v\rangle_\la|$, there hold
	\begin{equation*}
		\begin{aligned}
			|x|^{-1}\ast U_{x_{\la},p_\va} ^2 
			\leq C \la^{\frac{2}{p_\va-2}} \int_{\r^3} \frac{1}{|x-y|}Q_{p_\va}^2(\sqrt{\la}(y-x_{\la})) \dy
			\leq C \la^{\frac{2}{p_\va-2}-1},
		\end{aligned} 
	\end{equation*}
	\begin{equation*}
		\begin{aligned}
			 \int_{\mathbb{R}^{3}} 
			\big(|x|^{-1}\ast U_{x_{\la},p_\va} ^2\big) u v \dx
			\leq C \la^{\frac{2}{p_\va-2}-1} \int_{\r^3}uv \dx
			\leq C \la^{\frac{2}{p_\va-2}-2} ||u||_{\la}||v||_{\la},
		\end{aligned} 
	\end{equation*}
	\begin{equation*}
		\begin{aligned}
			\int_{\mathbb{R}^{3}} 
			 \big(|x|^{-1}\ast (U_{x_{\la},p_\va} u)\big)  U_{x_{\la},p_\va}  v \dx
			={}& \int_{\r^3}\int_{\r^3}\frac{1}{|x-y|}U_{x_{\la},p_\va} (x)u(x)U_{x_{\la},p_\va} (y)v(y) \dx\dy  \\
			\leq {}& C \|U_{x_{\la},p_\va} u\|_{L^{\frac65}(\r^3)}  \|U_{x_{\la},p_\va} v\|_{L^{\frac65}(\r^3)}  \\
			\leq {}& C \|U_{x_{\la},p_\va} \|_{L^{2}(\r^3)}^2 \|u\|_{L^{3}(\r^3)} \|v\|_{L^{3}(\r^3)} \\
			\leq{}& C \la^{\frac{2}{p_\va-2}-2} ||u||_{\la}||v||_{\la},
		\end{aligned} 
	\end{equation*}
thanks to the Hardy-Littlewood-Sobolev inequality \eqref{HLS}. So we have 
\[|\langle L_{2,\la}u,v\rangle_\la| \leq C \la^{\frac{2}{p_\va-2}-2} ||u||_{\la}||v||_{\la}
\leq C||u||_{\la}||v||_{\la}.\]
Hence, $L_\la$ is bounded.
\end{proof}

To solve equation \eqref{2-1-0}, the key is to prove $L_\lambda$ is invertible in $E_{\lambda ,p_\va}$.

\begin{lemma}\label{L-1} There exist $\la_{0}>0$ and $\va_0>0$    such that for any $\la > \la_{0}$, $x_{\la}\in B_\theta(b_0)$,  and $\va\in (0, \va_0)$,
it holds
\begin{equation}\label{2-7}
\|L_
 \lambda\xi\|_{\la} \geq \rho \|\xi\|_{\la}, \quad
 \text{ for  any } \xi \in E_{\lambda ,p_\va},
 \end{equation}
where $\rho> 0$ is a constant independent of $\la$ and $\va$.
\end{lemma}

\begin{proof}
	We argue by contradiction. Suppose that there are $\va_n\to 0$,
$\la_{n} \to \infty$, $x_{\la, n}\in B_\theta(b_0)$,
   $\xi_{n} \in E_{\la_{n},p_{\va_n}}$ and $\|\xi_{n}\|_{\la_{n}}=\lambda_n^{-\frac 14}$, such that
\begin{equation}\label{20-19-5}	
\|L_{\la_{n}} \xi_{n}\|_{\la_{n}} = o_{n}(1) \lambda_n^{-\frac 14}.
\end{equation}
By the proof of Lemma \ref{lemma1}, we have
\[
|\langle L_{2,\la_n}\xi,\varphi\rangle_\la| \leq o_n(1) ||\xi_n||_{\la}||\varphi||_{\la},
\]
for any $\varphi \in E_{\la, p_\va}$. Thus \eqref{20-19-5} is equivalent to 
\begin{equation}\label{20-19-6}	
	\|L_{1,\la_{n}} \xi_{n}\|_{\la_{n}} = o_{n}(1) \lambda_n^{-\frac 14}.
\end{equation}
It is a bit standard to obtain a contradiction from \eqref{20-19-6}.
So we just sketch the proof. 	For simplicity, we drop the subscript $n$.

	For any $\varphi \in E_{\la, p_\va}$, we have
	\begin{equation}\label{2-1-1}
		o(1)\|\varphi\|_{\la}\lambda^{-\frac 14}= \int_{\mathbb{R}^{3}}\Bigl(\nabla \xi \nabla \varphi + \la \xi \varphi - (p_{\va}-1)V(x) U_{x_{\la}, p_{\va}}^{p_{\va}-2} \xi \varphi  \Bigr) \dx.
	\end{equation}
	Letting $\varphi = \xi$ in \eqref{2-1-1},  we are led to
	\begin{equation}\label{2-1-3}
		\lambda^{ -\frac 12}  -  (p_{\va} -1)\int_{\mathbb{R}^{3}}V(x)\Big[ \Big(\frac{\la}{V_0}\Big)^{\frac{1}{p_{\va}-2}}Q_{p_\va}\big(\sqrt{\la} (x -x_{\la})\big)\Big]^{p_{\va}-2}\xi^{2}  \dx = o(1)\lambda^{ -\frac 12}.
	\end{equation}

	From \eqref{2-1-3} and Lemma~\ref{4-1}, we will obtain the contradiction, if we can prove
	\begin{equation}\label{2-1-4}
		\lambda \int_{B_{R/\sqrt{\la}}(x_{\la})} \xi^{2} \dx = o(\lambda^{-\frac 12}).
	\end{equation}
	To prove \eqref{2-1-4}, we  define
	\[\bar\xi_{\la}(x) = \xi \Big(\frac{x}{\sqrt{\la}}+x_{\la} \Big). \]
	Then
\[
\int_{\mathbb R^3} \bigl( |\nabla \bar\xi_{\la}|^2 + \bar\xi_{\la}^2\bigr) \dx \le
C.
\]
Thus,
there exists $\bar \xi \in H^{1}(\mathbb{R}^{3})$, such that, as $\la \to \infty$,
	\[ \bar \xi_{\la}\rightharpoonup \bar\xi \text{ weakly in } H^{1}(\mathbb{R}^{3}), \]
	and
	\[\bar \xi_{\la} \to \bar \xi \text{ strongly in } L^{2}_{loc}(\mathbb{R}^{3}). \]
	It is easy to see that  \eqref{2-1-4} is equivalent to $\bar \xi = 0 $.
	
	Using Lemma~\ref{4-1}, we can prove  that $\bar \xi$ satisfies the following equation
	\begin{equation*}
		-\Delta \bar \xi + \bar \xi -  (\bar {p} -1) Q_{\bar{p}}^{\bar {p}-2} \bar\xi =0.
	\end{equation*}
	Hence, by the non-degeneracy of the solution $Q_{\bar{p}}$,
	\[\bar \xi = \sum_{j=1}^{3} d_{j} \frac{\partial Q_{\bar{p}}}{\partial x_{j}}.\]
	
	On the other hand, from  $\xi \in E_{\la, p_{\va}}$,  we obtain
	\begin{equation}\label{2-1-23}\begin{split}
			0 = &  \langle \xi, \frac{\partial U_{x_{\la}, p_{\va}}}{\partial x_{j}} \rangle_{\la}
			= (  p_{\va} -1)  \Big(\frac{\la}{V_0}\Big)^{\frac{1}{p_{\va}-2}} \biggl[ \int_{\mathbb{R}^{3}} Q^{p_{\va}-2}_{p_{\va}}(x) \frac{\partial Q_{p_\va}( x )}{\partial x_{j}} \bar \xi_{\la}(x) \dx +o(1)\biggr].
		\end{split}
	\end{equation}
	This gives
	\begin{equation}\label{2-1-24}
		\int_{\mathbb{R}^{3}} Q^{p_{\va}-2}_{p_{\va}}(x) \frac{\partial Q_{p_{\va}}( x )}{\partial x_{j}} \bar \xi_{\la}(x) \dx = o(1),
	\end{equation}
	  which together with Lemma~\ref{4-1} implies that
	\begin{equation}\label{2-1-25}
		\int_{\mathbb{R}^{3}} Q_{\bar{p}}^{\bar{p}-2}(x) \frac{\partial Q_{\bar{p}}( x )}{\partial x_{j}} \bar \xi(x) \dx = 0.
	\end{equation}
	Thus, $d_{j} = 0$ for $j=1,2, 3$.
	The proof is completed.
\end{proof}

From now on, we always assume that $\va\in (0, \va_0)$. Using Lemma~\ref{4-1}, we can prove the following
 lemmas which give the estimates of $l_{\la}$ and $R_{\la}(\omega)$.

\begin{lemma}\label{2-2}  There is a constant $C>0$, independent of $\va\in (0, \va_0)$, such that
\[\|l_{\la}\|_{\la} \leq  C  \Bigl( | x_{\la}-b_0|^2 \la^{\frac{1}{p_{\va}-2} -\frac {1}{4}}+ | \nabla V(x_{\la})| \la^{\frac{1}{p_{\va}-2} -\frac {3}{4}}  +   \la^{\frac{3}{p_\va-2}-\frac94}\Bigr) .\]
\end{lemma}
\begin{proof}
	Recall that, for any $\varphi \in H^{1}(\mathbb{R}^{3})$,
	\[
	\aligned\langle l_{\la}, \varphi\rangle_{\la}
	=& \Big\{\int_{\mathbb{R}^{3} \setminus B_{\delta}(x_{\la}) }+ \int_{B_{\delta}(x_{\la})}\Big\} (V(x)-V_0) U_{x_{\la}, p_{\va}}^{p_\va-1} \varphi \dx
	- \int_{\r^3}\big(|x|^{-1}\ast U_{x_{\la},p_\va} ^2\big)U_{x_{\la},p_\va} \varphi \dx.
	\endaligned\]
	It is easy to get
	\begin{equation}\label{2-2-01}
		\Big|\int_{\mathbb{R}^{3} \setminus B_{\delta}(x_{\la}) }(V(x)-V_0) U_{x_{\la}, p_{\va}}^{p_\va-1} \varphi  \dx \Big|
		\leq C e^{-\tau\sqrt{\la}} \|\varphi\|_\la, 
	\end{equation}
for some $\tau>0$ small.
	By H\"older inequality,  one has
	\begin{equation}\label{2-2-1}\begin{split}
			& \Big|\int_{B_{\delta}(x_{\la})}( V(x)-V_0) U_{x_{\la}, p_{\va}}^{p_\va-1}  \varphi \dx \Big|  \\
			= {}& \Big|\int_{B_{\delta}(x_{\la})} \bigl((V(x_{\la})-V_0) + \nabla V(x_{\la}) \cdot (x-x_{\la}) + O(|x-x_{\la}|^{2}) \bigr) U_{x_{\la}, p_{\va}}^{p_\va-1}  \varphi  \dx \Big| \\
			\leq {}& C|V(x_{\la})-V_{0}| \Bigl( \int_{\mathbb{R}^{3}} U_{x_{\la}, p_{\va}}^{p_\va} \dx\Bigr)^{\frac{p_\va-1}{p_\va}} \cdot   \Big( \int_{\mathbb{R}^{3}} |\varphi|^{p_\va} \dx\Big)^{\frac{1}{p_\va}} + C |\nabla V(x_{\la})| \la^{\frac{1}{p_{\va}-2} -\frac 34} \|\varphi\|_{\la}+ O\Bigl(  \la^{\frac{1}{p_{\va}-2} -\frac54} \Bigr) \|\varphi\|_{\la} \\
			\leq {}&C  \Bigl( |x_{\la}-b_0|^{2} \la^{\frac{1}{p_{\va}-2} -\frac {1}{4}} +|\nabla V(x_{\la})| \la^{\frac{1}{p_{\va}-2} -\frac 34}  +  \la^{\frac{1}{p_{\va}-2} -\frac54}\Bigr)  \|\varphi\|_{\la} .
		\end{split}
	\end{equation}
	Moreover, by the Hardy-Littlewood-Sobolev inequality \eqref{HLS}, we have
	\begin{equation}\label{eq2.26}
		\aligned
		\Big|\int_{\r^3}\big(|x|^{-1}\ast U_{x_{\la},p_\va} ^2\big)U_{x_{\la},p_\va} \varphi \dx \Big|
		={}& \int_{\r^3}\int_{\r^3} \frac{1}{|x-y|}U_{x_{\la},p_\va} (x)\varphi(x)U_{x_{\la},p_\va} ^2(y) \dx\dy \\
		\leq{}& C \|U_{x_{\la},p_\va} \|_{L^{\frac{12}5}(\r^3)}^2 \|U_{x_{\la},p_\va} \varphi\|_{L^{\frac65}(\r^3)}\\
		\leq{}& C \la^{\frac{2}{p_\va-2}-\frac54} \|U_{x_{\la},p_\va} \|_{L^{2}(\r^3)} \|\varphi\|_{L^{3}(\r^3)} \\
		\leq{}& C \la^{\frac{3}{p_\va-2}-\frac94} \|\varphi\|_\la.
		\endaligned
	\end{equation}
The conclusion of the lemma follows from \eqref{2-2-1} and \eqref{eq2.26}.
\end{proof}

\begin{lemma}\label{2-3} There is a constant $C>0$, independent of $\va\in (0, \va_0)$, such that
\[\|R_{\la}(\omega)\|_{\la} \leq   C\biggl( \frac{\|\omega\|^{2}_{\la} }{\la^{\frac{1}{p_{\va} -2}-\frac{1}{4}}}  + \frac{\|\omega\|_{\la}^{p_{\va}-1}}{\la^{\frac{3}{2}-\frac{p_{\va}}{4}}}  +\frac{\|\omega\|_{\la}^{3}}{\la^{\frac{3}{2}}}  \biggr) . \]
\end{lemma}
\begin{proof}
	Let $ \langle R_{1,\la}(\omega),\varphi\rangle_\la=\int_{\mathbb{R}^3}V(x)\big((U_{x_{\la},p_\va} +\omega)^{p_\va-1}-U_{x_{\la},p_\va} ^{p_\va-1}-
	(p_\va-1)U_{x_{\la},p_\va} ^{p_\va-2}\omega\big)\varphi\dx$ and $\langle R_{2,\la}(\omega),\varphi\rangle_\la=\langle R_{\la}(\omega),\varphi\rangle_\la-\langle R_{1,\la}(\omega),\varphi\rangle_\la$, for any $\varphi \in E_{\la, p_\va}$.
	By H\"older inequality and \eqref{phi}, it is standard to show that
	\begin{equation}\label{R1}
		\|R_{1,\la}(\omega)\|_{\la} \leq   C\biggl( \frac{\|\omega\|^{2}_{\la} }{\la^{\frac{1}{p_{\va} -2}-\frac{1}{4}}}  + \frac{\|\omega\|_{\la}^{p_{\va}-1}}{\la^{\frac{3}{2}-\frac{p_{\va}}{4}}}  \biggr) .
	\end{equation}
For $R_{2,\la}(\omega)$, we have
	\begin{equation*}
		\langle R_{2,\la}(\omega),\varphi\rangle_\la = 2 \int_{\r^3}\big(|x|^{-1}\ast (U_{x_{\la},p_\va} \omega)\big) \omega \varphi \dx
		+ \int_{\r^3}\big(|x|^{-1}\ast \omega^2\big) \omega \varphi \dx
		, \quad \forall \varphi \in E_{\lambda ,p_\va}.
	\end{equation*}
	By the Hardy-Littlewood-Sobolev inequality \eqref{HLS}, there hold
	\begin{equation*}
		\aligned
		\Big|\int_{\r^3}\big(|x|^{-1}\ast  (U_{x_{\la},p_\va} \omega)\big)\omega\varphi \dx\Big|
		=& \int_{\r^3}\int_{\r^3} \frac{1}{|x-y|}\omega(x)\varphi(x)U_{x_{\la},p_\va} (y)\omega(y) \dx\dy \\
		\leq& C  \|U_{x_{\la},p_\va} \|_{L^{2}(\r^3)} \|\omega\|_{L^{3}(\r^3)} \|\omega\|_{L^{\frac32}(\r^3)} \|\varphi\|_{L^{6}(\r^3)} \\
		\leq& C \la^{\frac{1}{p_\va-2}-\frac74} \|\omega\|_\la^2\|\varphi\|_\la ,
		\endaligned
	\end{equation*}
	and
	\begin{equation*}
		\aligned
		\Big|\int_{\r^3}\big(|x|^{-1}\ast \omega^2\big)\omega\varphi \dx\Big|
		=& \int_{\r^3}\int_{\r^3} \frac{1}{|x-y|}\omega(x)\varphi(x)\omega^2(y) \dx\dy \\
		\leq& C \|\omega\|_{L^{\frac{12}5}(\r^3)}^2 \|\omega\|_{L^{\frac32}(\r^3)} \|\varphi\|_{L^{6}(\r^3)} \\
		\leq& C \la^{-\frac32} \|\omega\|_\la^3 \|\varphi\|_\la .
		\endaligned
	\end{equation*}
	Thus
	\begin{equation}\label{R2}
		\|R_{2,\la}(\omega)\|_{\la} \leq   C\biggl( \frac{\|\omega\|^{2}_{\la} }{\la^{-\frac{1}{p_{\va} -2}+\frac{7}{4}}}  + \frac{\|\omega\|_{\la}^{3}}{\la^{\frac{3}{2}}}  \biggr) .
	\end{equation}
			Hence, the result follows from \eqref{R1}-\eqref{R2}.
		\end{proof}

 Lemma \ref{L-1} indicates $L_\la$ is invertible in $E_{\la, p_\va}$. Lemma \ref{2-3} shows that problem \eqref{2-1-0} is a perturbation of linear problem $L_\la\omega=l_\la$. Therefore combining the above Lemmas \ref{L-1}, \ref{2-2}, \ref{2-3} and the contraction mapping theorem yields the following result.

\begin{proposition}\label{2-4} There exist constants $\va_0>0$ and  $\lambda_0>0$, such that for
any $\lambda>\lambda_0$ and $0<\va<\va_0$, there exists a $C^{1}$ function $\omega_{\la}$  from $B_\theta(b_0)$ to $ E_{\la, p_\va}$ satisfying
\begin{equation}\label{reduction}
  L_{\la} \omega_{\la} = l_{\la} + R_{\la}(\omega_{\la})\; \text{ in } E_{\la, p_\va}.
\end{equation}
Moreover,
\[\|\omega_{\la}\|_{\la}\leq C \Bigl( | x_{\la}-b_0|^2 \la^{\frac{1}{p_{\va}-2} -\frac {1}{4}}+ | \nabla V(x_{\la})| \la^{\frac{1}{p_{\va}-2} -\frac {3}{4}}  +   \la^{\frac{3}{p_\va-2}-\frac94}\Bigr) . \]
\end{proposition}


So far, the problem \eqref{lambda} is reduced to a finite dimensional problem. We will find suitable $x_\la$ such that \eqref{reduction} holds in $H^1(\r^3)$.

\begin{proof}[\bf Proof of Theorem \ref{ExistenceWithoutNormalized}.]
For any $x_{\la}\in B_\theta(b_0)$,   let
 $u_{\la} = U_{x_{\la}, p_{\va}} + \omega_\la$, where $\omega_\la$ is the function obtained in Proposition~\ref{2-4}. Then we have 
 \[
 L_{\la} \omega_{\la} = l_{\la} + R_{\la}(\omega_{\la}) + \sum_{j=1}^{3}c_{\la,j}\frac{\partial U_{x_\la, p_\va}}{\partial x_j} \, \;\text{ in } H^1(\r^3),
 \]
 for some constants $c_{\la,j}$.
 To obtain a solution for $u_\la$ of \eqref{lambda},  we need to find a suitable $x_\la$ such that $c_{\la,j}=0$, which is equivalent to solving the following equations 
  \begin{equation}\label{2-5-2}
\begin{split}
 \int_{\mathbb{R}^{3}} \biggl(  \nabla u_{\la} \nabla \frac{\partial U_{x_{\la}, p_{\va}}}{\partial x_{j}}  
 +\la u_{\la} \frac{\partial U_{x_{\la}, p_{\va}} }{\partial x_{j}} 
 - V(x)u_{\la}^{p_{\va}-1} \frac{\partial U_{x_{\la}, p_{\va}}}{\partial x_{j}} 
 +  (|x|^{-1} \ast |u_{\la}|^{2})u_{\la} \frac{\partial U_{x_{\la}, p_{\va}}}{\partial x_{j}} &\biggr) \dx
 =0, \\
 & j=1,2, 3.
\end{split}
\end{equation}

Since $\omega_{\la} \in E_{\la}$, for $j =1, 2,3$, one has
\begin{equation}\label{2-5-3}\begin{split}
		& \int_{\mathbb{R}^{3}} \biggl(  \nabla u_{\la} \nabla \frac{\partial U_{x_{\la}, p_{\va}}}{\partial x_{j}}  +\la u_{\la} \frac{\partial U_{x_{\la}, p_{\va}}}{\partial x_{j}}- V(x) u_{\la}^{p_{\va}-1} \frac{\partial U_{x_{\la}, p_{\va}}}{\partial x_{j}}\biggr) \dx \\
		= & \int_{\mathbb{R}^{3}} \big(V_0-V(x)\big) U_{x_{\la}, p_{\va}}^{p_\va-1} \frac{\partial U_{x_{\la}, p_{\va}}}{\partial x_{j}} \dx
		- \int_{\mathbb{R}^{3}} V(x) \Bigl[ (U_{x_{\la},p_{\va}} +\omega_{\la})^{p_{\va}-1} - U_{x_{\la}, p_{\va}}^{p_{\va}-1}\Bigr] \frac{\partial U_{x_{\la}, p_{\va}}}{\partial x_{j}} \dx .
	\end{split}
\end{equation}
According to the integration by parts, we have
\begin{equation}\label{2-5-4}\begin{split}
		\int_{\mathbb{R}^{3}} \bigl(V_{0}-V(x)\bigr) U_{x_{\la}, p_{\va}}^{p_\va-1} \frac{\partial U_{x_{\la}, p_{\va}}}{\partial x_{j}} \dx
		= {}& \frac1{p_\va}\int_{\mathbb{R}^{3}} \frac{\partial V(x)}{\partial x_{j}} U^{p_\va}_{x_{\la}, p_{\va}} \dx \\
		={}& \frac{1}{p_\va}\frac{\la^{\frac{2}{p_\va-2}-\frac12}}{ V_{0}^{\frac{2}{p_{\va}-2}+1}} \frac{\partial V(x_{\la})}{\partial x_{j}} \int_{\mathbb{R}^{3}}Q_{p_{\va}}^{p_\va}(x) \dx 
		+ O\Bigl( \la^{\frac{2}{p_\va-2}-1}\Bigr).
	\end{split}
\end{equation}
It follows from H\"older inequality and \eqref{phi} that
\begin{equation}\label{2-5-5}\begin{split}
		& \int_{\mathbb{R}^{3}} V(x) \Bigl[ ( U_{x_{\la}, p_{\va}} +\omega_{\la})^{p_{\va}-1} - U_{x_{\la}, p_{\va}}^{p_{\va}-1}\Bigr] \frac{\partial U_{x_{\la}, p_{\va}}}{\partial x_{j}} \dx \\
		= {}& (p_\va-1)\int_{\mathbb{R}^{3}} (V(x)-V(x_{\la})) U_{x_{\la}, p_{\va}}^{p_{\va}-2} \omega_{\la}\frac{\partial U_{x_{\la}, p_{\va}}}{\partial x_{j}} \dx \\
		{}&+O\biggl( \int_{\mathbb{R}^{3}} U^{p_{\va}-3}_{x_{\la}, p_{\va}} |\omega_{\la}|^{2}  \bigg|\frac{\partial U_{x_{\la}, p_{\va}}}{\partial x_{j}}\bigg| \dx 
		+  \int_{\mathbb{R}^{3}}|\omega_{\la}|^{p_{\va}-1}  \bigg|\frac{\partial U_{x_{\la}, p_{\va}}}{\partial x_{j}}\bigg|  \dx  \biggr)\\
		={} & O\Bigl( |\nabla V(x_{\la})|\la^{\frac{1}{p_\va-2}-\frac14}\|\omega_{\la}\|_{\la} + \la^{\frac12} \|\omega_{\la}\|_{\la}^{2} + \la^{\frac{1}{p_\va-2}-\frac{p_\va}{4}-\frac34} \|\omega_{\la}\|_{\la}^{p_{\va}} \Bigr) .
	\end{split}
\end{equation}
Notice that
\[
\begin{split}
 &\int_{\mathbb{R}^{3}}   \big(|x|^{-1} \ast |U_{x_{\la}, p_{\va}}|^{2}\big)U_{x_{\la}, p_{\va}} \frac{\partial U_{x_{\la}, p_{\va}}}{\partial x_{j}}  \dx 	
 = \frac{1}{2} \int_{\r^3}\int_{\r^3} \frac{x_j-y_j}{|x-y|^3} U_{x_{\la}, p_{\va}}^{2} (x) U_{x_{\la}, p_{\va}}^{2} (y) \dx\dy,
\end{split}
\]
and 
\[
\begin{split}
	\int_{\mathbb{R}^{3}}   \big(|x|^{-1} \ast |U_{x_{\la}, p_{\va}}|^{2}\big)U_{x_{\la}, p_{\va}} \frac{\partial U_{x_{\la}, p_{\va}}}{\partial x_{j}}  \dx 	
	=& \int_{\mathbb{R}^{3}}   \Big(|x|^{-1} \ast U_{x_{\la}, p_{\va}} \frac{\partial U_{x_{\la}, p_{\va}}}{\partial x_{j}} \Big) |U_{x_{\la}, p_{\va}}|^{2} \dx \\	
	=& \frac{1}{2} \int_{\r^3}\int_{\r^3} \frac{y_j-x_j}{|x-y|^3} U_{x_{\la}, p_{\va}}^{2} (x) U_{x_{\la}, p_{\va}}^{2} (y) \dx\dy.
\end{split}
\]
Then we derive that
\begin{equation}\label{symetry}
\begin{split}
	\int_{\mathbb{R}^{3}}   (|x|^{-1} \ast |U_{x_{\la}, p_{\va}}|^{2})U_{x_{\la}, p_{\va}} \frac{\partial U_{x_{\la}, p_{\va}}}{\partial x_{j}}  \dx 	
	= 0.
\end{split}
\end{equation}
Thus, we obtain 
  \begin{equation}\label{2-5-6}
	\begin{split}
		& \int_{\mathbb{R}^{3}}   \big(|x|^{-1} \ast |u_{\la}|^{2}\big)u_{\la} \frac{\partial U_{x_{\la}, p_{\va}}}{\partial x_{j}}  \dx  \\
		={}&  \int_{\mathbb{R}^{3}}   \big(|x|^{-1} \ast |U_{x_{\la}, p_{\va}}|^{2}\big)U_{x_{\la}, p_{\va}} \frac{\partial U_{x_{\la}, p_{\va}}}{\partial x_{j}}  \dx 
		+ \int_{\mathbb{R}^{3}}   \big(|x|^{-1} \ast |U_{x_{\la}, p_{\va}}|^{2}\big)\omega_{\la} \frac{\partial U_{x_{\la}, p_{\va}}}{\partial x_{j}}  \dx  \\ 
		{}& +2\int_{\mathbb{R}^{3}}   \big(|x|^{-1} \ast |U_{x_{\la}, p_{\va}}\omega_{\la}|\big)U_{x_{\la}, p_{\va}}  \frac{\partial U_{x_{\la}, p_{\va}}}{\partial x_{j}}  \dx  
		+2\int_{\mathbb{R}^{3}}   \big(|x|^{-1} \ast |U_{x_{\la}, p_{\va}}\omega_{\la}|\big)\omega_{\la}  \frac{\partial U_{x_{\la}, p_{\va}}}{\partial x_{j}}  \dx \\
		{}& +\int_{\mathbb{R}^{3}}   \big(|x|^{-1} \ast |\omega_{\la}^2|\big)U_{x_{\la}, p_{\va}}  \frac{\partial U_{x_{\la}, p_{\va}}}{\partial x_{j}}  \dx  
		+\int_{\mathbb{R}^{3}}   \big(|x|^{-1} \ast |\omega_{\la}^2|\big)\omega_{\la}  \frac{\partial U_{x_{\la}, p_{\va}}}{\partial x_{j}}  \dx \\
		={}& O\Big(\la^{\frac{3}{p_\va-2}-\frac74} \|\omega_{\la}\|_\la
		+ \la^{\frac{2}{p_\va-2}-\frac32} \|\omega_{\la}\|_\la^2 
		+\la^{\frac{1}{p_\va-2}-\frac54} \|\omega_{\la}\|_\la^3 \Big).
	\end{split}
\end{equation}
Using the estimates in Lemma \ref{4-1} and Proposition \ref{2-4}, we immediately obtain that  \eqref{2-5-2} is equivalent to
\begin{equation}\label{nn2-5-6}
\frac{\partial V(x_{\la})}{\partial x_{j}} = O\bigl(\la^{-\frac12}\bigr), \,j=1,2, 3.
\end{equation}
Since $b_0$ is a non-degenerate point of $V(x)$, \eqref{nn2-5-6} has a solution $x_{\la}$ with $|x_{\la} - b_0| = O\bigl(\la^{-\frac12}\bigr)$ .
\end{proof}

\subsection{Local Uniqueness}

In the last subsection, we study the existence of peak solution $u_\la$ for \eqref{lambda} of the form
\begin{equation}\label{20-21-5}
u_{\la} = U_{x_{\la},p_\va} + \omega_\la,
\end{equation}
where  as $\lambda\to +\infty$,
\begin{equation}\label{10-21-5}
x_{\la}\to b_0, \quad\int_{\mathbb{R}^{3}} \bigl(|\nabla \omega_\la|^{2} + \la  \omega_\la^{2} \bigr) \dx =o\bigl(\lambda^{\frac{2}{p_\va-2}-\frac 12}\bigr).
\end{equation}
Lemmas~\ref{L-1}, \ref{2-2} and \ref{2-3} show that
for any single-peak solution $u_\la$
of the form \eqref{20-21-5}, it holds
\begin{equation}\label{21-21-5}
|x_{\la} - b_0| = O(\la^{-\frac12}),\quad \|\omega_{\la}\|_{\la}\leq C   \la^{\frac{3}{p_{\va}-2} -\frac{9}{4}} .
\end{equation}

We now study  the local uniqueness of $u_{\la}$ of the form \eqref{20-21-5}, satisfying \eqref{21-21-5}.

\begin{theorem}\label{3-7}
Suppose (V) holds. Then
there exist constants $\va_0>0$ and  $\lambda_0>0$, such that for
any $\lambda>\lambda_0$ and $0<\va<\va_0$, if
 \eqref{lambda} has   solutions $u^{(1)}_{\la}$ and $u_{\la}^{(2)}$ of the form \eqref{20-21-5}, satisfying \eqref{21-21-5}, then $u^{(1)}_{\la}=u_{\la}^{(2)}$.
\end{theorem}

\begin{remark}
For fixed $\va>0$ and sufficiently large $\lambda$, the local uniqueness result can be easily proved by applying the same strategy in \cite{Li2021}.  In the above theorem,
we prove the local uniqueness result for $\lambda>\lambda_0$, where $\lambda_0>0$ is independent of $\va\in (0, \va_0)$.
\end{remark}

The main tool to prove the local uniqueness  is the following local Pohozaev identity
for  a solution $u$ of \eqref{lambda}.

\begin{proposition}
	It holds that
	\begin{equation}\label{Pohozaev}\begin{split}
			&\frac{1}{p_\va} \int_{ B_{d}(x_{\la})} \frac{\partial V(x)}{\partial x_j} u^{p_\va} \dx  \\ 
			=  {}& \int_{\partial B_{d}(x_{\la})} \frac{\partial u}{\partial \nu} \frac{\partial u}{\partial x_{j}} \dsi
			-\frac12  \int_{\partial B_{d}(x_{\la})} |\nabla u|^2 \nu_j \dsi
			- \frac{\la }{2}\int_{\partial B_{d}(x_{\la})} u^2 \nu_j \dsi
			+ \frac{1}{p_\va} \int_{\partial B_{d}(x_{\la})} V(x) u^{p_\va} \nu_j \dsi  \\
			{}&- \frac12 \int_{\partial B_{d}(x_{\la})}\int_{\r^3} \frac{u^2(y)}{|x-y|} \dy \cdot u^2(x) \nu_j \dsi 
			+\frac12 \int_{B_{d}(x_{\la})}\int_{\r^3} \frac{x_j-y_j}{|x-y|^3} u^2(y) \dy \cdot u^2(x) \dx , 
		\end{split}
	\end{equation}
	where $\nu$ is the unit outward normal to  $\partial B_{d}(x_{\la})$.
\end{proposition}

\begin{proof}
	Multiplying \eqref{lambda} by $\frac{\partial u}{\partial x_j}$ and integrating in $B_{d}(x_{\la})$, we have
	\begin{equation}\label{Pohozaev1}
		\int_{B_{d}(x_{\la})} -\Delta u \frac{\partial u}{\partial x_{j}} \dx 
		+ \la \int_{B_{d}(x_{\la})} u  \frac{\partial u}{\partial x_{j}} \dx
		+ \int_{B_{d}(x_{\la})} \big(|x|^{-1}\ast u^2\big) u \frac{\partial u}{\partial x_{j}}  \dx
		=  \int_{ B_{d}(x_{\la})} V(x) u^{p_\va-1}\frac{\partial u}{\partial x_{j}} \dx .
	\end{equation}
	It is easy to check that
	\begin{equation}\label{Pohozaev2}
		\begin{split}
			&\int_{B_{d}(x_{\la})} -\Delta u \frac{\partial u}{\partial x_{j}} \dx 
			+ \la \int_{B_{d}(x_{\la})} u  \frac{\partial u}{\partial x_{j}} \dx
			+ \int_{B_{d}(x_{\la})} \big(|x|^{-1}\ast u^2\big) u \frac{\partial u}{\partial x_{j}} \dx  \\
			={}& - \int_{\partial B_{d}(x_{\la})} \frac{\partial u}{\partial \nu} \frac{\partial u}{\partial x_{j}} \dsi 
			+\frac12  \int_{\partial B_{d}(x_{\la})} |\nabla u|^2 \nu_j \dsi
			+ \frac{\la }{2}\int_{\partial B_{d}(x_{\la})} u^2 \nu_j \dsi \\
			{}&+ \frac12 \int_{\partial B_{d}(x_{\la})}\int_{\r^3} \frac{u^2(y)}{|x-y|} \dy \cdot u^2(x) \nu_j \dsi 
			-\frac12 \int_{B_{d}(x_{\la})}\int_{\r^3} \frac{x_j-y_j}{|x-y|^3} u^2(y) \dy \cdot u^2(x) \dx
		\end{split}
	\end{equation}
	and 
	\begin{equation}\label{Pohozaev3}
		\begin{split}
			\int_{ B_{d}(x_{\la})} V(x) u^{p_\va-1}\frac{\partial u}{\partial x_{j}} \dx 
			= \frac{1}{p_\va} \int_{\partial B_{d}(x_{\la})} V(x) u^{p_\va} \nu_j \dsi 
			- \frac{1}{p_\va} \int_{ B_{d}(x_{\la})} \frac{\partial V(x)}{\partial x_j} u^{p_\va} \dx.
		\end{split}
	\end{equation}
	Then the result follows from \eqref{Pohozaev1}-\eqref{Pohozaev3}.
\end{proof}

To use \eqref{Pohozaev}, we need to estimate the solutions $u_\la$ of \eqref{lambda}, having the properties
stated in Theorem~\ref{ExistenceWithoutNormalized}.

\begin{lemma}\label{3-1} Fix $d>0$ small. For any $\theta>0$ small,
there exists  $C>0$, independent of $\va$, such that for any $\va\in (0, \va_0)$, we have
\[|u_{\la}(x)| \leq C e^{-(1-\theta) \sqrt{\la}|x- x_{\la}|},\quad
 \text{ for all }\,  x \in \mathbb{R}^{3},\]
and
\[|\nabla u_{\la}(x)| \leq C e^{-(1-\theta) \sqrt{\la}},
\quad \text{ for all } \, x\in \partial B_{d}(x_{\la}) . \]
\end{lemma}

\begin{proof}
 The proof is standard by using the comparison principle. See for example the proof of Lemma~2.1.2 and Lemma~3.1.1 in \cite{CPY}. We thus omit it.
\end{proof}

We now improve the estimate for $|x_{\la} - b_0|$.

\begin{lemma}\label{3-3} If $u_{\la}$ is the solution of \eqref{lambda} of the form \eqref{2}, then there exists a constant $\va_0>0$ such that $x_{\la}$ satisfies
\[|x_{\la} - b_{0}| = o(\la^{-\frac12}), \]
for any $\va\in(0,\va_0)$.
\end{lemma}
\begin{proof} By Lemma \ref{3-1} and the symmetry, we see that
\begin{equation}\label{3-3-1}
	\begin{split}
		\text{RHS of \eqref{Pohozaev}} ={}& \frac12 \int_{B_{d}(x_{\la})}\int_{\r^3} \frac{x_j-y_j}{|x-y|^3} u_\la^2(y) \dy \cdot u_\la^2(x) \dx + O(e^{-\tau \sqrt{\la}})\\
		={}& -\frac12 \int_{\r^3\setminus B_{d}(x_{\la})}\int_{\r^3} \frac{x_j-y_j}{|x-y|^3} u_\la^2(y) \dy \cdot u_\la^2(x) \dx + O(e^{-\tau \sqrt{\la}})\\
		={}& O(e^{-\tau \sqrt{\la}}),
	\end{split}
\end{equation}
for some $\tau>0$ small. Then, \eqref{Pohozaev} is equivalent to 
\begin{equation}\label{3-3-1-0}\begin{split}
 \int_{B_{d}(x_{\la})} \frac{\partial V(x)}{\partial x_{j}} u_{\la}^{p_\va} \dx =O(e^{-\tau \sqrt{\la}}).
\end{split}
\end{equation}
By the Taylor's expansion, we have
\begin{equation}
	\begin{split}
		  &\int_{B_{d}(x_{\la})} \frac{\partial V(x)}{\partial x_{j}} u_{\la}^{p_\va} \dx  \\
		 ={}& \frac{\partial V(x_{\la})}{\partial x_{j}} \int_{B_{d}(x_{\la})}  u_{\la}^{p_\va} \dx 
		 +  \int_{B_{d}(x_{\la})}  \langle \nabla \frac{\partial V(x_{\la})}{\partial x_{j}}, x -x_{\la} \rangle u_{\la}^{p_\va} \dx + O\Bigl(  \int_{B_{d}(x_{\la})}  |x-x_{\la}|^{2} u_{\la}^{p_\va} \dx \Bigr)\\ 
		 =:{}& I_1 +I_2 +O(\la^{\frac{2}{p_\va-2}-\frac32}).
	\end{split}
\end{equation}
By the direct computations, we obtain
\begin{equation}
	\begin{split}
		I_1 =& A V_0^{-\frac{p_\va}{p_\va-2}} \la^{\frac{2}{p\va-2}-\frac12}  \frac{\partial V(x_{\la})}{\partial x_{j}} ,
	\end{split}
\end{equation}
where
\begin{equation}
	A = \int_{\r^3} Q_{p_\va}^{p_\va} \dx >0
\end{equation}
is a constant.
It follows from  Proposition \ref{2-4} that
\begin{equation}\label{3-3-2}\begin{split}
I_2 
= {}&  \int_{B_{d}(x_{\la})}  \langle \nabla \frac{\partial V(x_{\la})}{\partial x_{j}}, x-x_{\la} \rangle \Bigl( U_{x_{\la}, p_{\va}}^{p_\va} + p_\va U_{x_{\la}, p_{\va}}^{p_\va-1} \omega_{\la} +O\big(U_{x_{\la}, p_{\va}}^{p_\va-2}\omega_{\la}^{2} +\omega_{\la}^{p_\va}\big) \Bigr) \dx  \\
= {}&  p_\va  \int_{B_{d}(x_{\la})} \langle \nabla \frac{\partial V(x_{\la})}{\partial x_{j}}, x-x_{\la} \rangle  U_{x_{\la}, p_{\va}}^{p_\va-1} \omega_{\la} \dx \\
{}&+O\Big( \int_{B_{d}(x_{\la})} |x-x_{\la}| U_{x_{\la}, p_{\va}}^{p_\va-2}\omega_{\la}^{2} \dx + \int_{B_{d}(x_{\la})} |x-x_{\la}| \omega_{\la}^{p_\va} \dx \Big)  \\
\leq{}& C \la^{\frac{1}{p_\va-2}-\frac34} \|\omega_{\la}\|_\la + C \la^{-\frac12} \|\omega_{\la}\|_\la^2 + C \la^{\frac{p_\va}{4}-2} \|\omega_{\la}\|_\la^{p_\va} \\
\leq{}& C  \la^{\frac{4}{p_\va-2}-3} .
 \end{split}
\end{equation}

 Combining \eqref{3-3-1-0}-\eqref{3-3-2}, we are led to
\begin{equation}\label{3-3-6}
 \frac{\partial V(x_{\la})}{\partial x_j}  = O\Bigl(\la^{-1}+\la^{\frac{2}{p_{\va}-2}-\frac52 } \Bigr) = o(\la^{-\frac12}).
\end{equation}
Since $b_0$ is a non-degenerate critical point of $V(x)$, we obtain
\[|x_{\la} - b_0| = o(\la^{-\frac12}).\]
The proof is completed.
\end{proof}

To prove the local uniqueness result, we  argue by contradiction. Suppose that
there are $\va_n\to 0$, $\la_n\to +\infty$, such that  \eqref{lambda}
has two different solutions $u^{(1)}_{\la_n}$ and $u_{\la_n}^{(2)}$ of the form
\[u_{\la_n}^{(j)}(x) =\Big(\frac{\la_n}{V_0}\Big)^{\frac{1}{p_{\va_n}-2}}Q_{p_{\va_n}}\big(\sqrt{\la_n}(x- x_{\la_n}^{(j)})\big)+ \omega^{(j)}_{\la_n},\,j=1,2, \]
where  $x_{\la_n}^{(j)} \to b_{0}$ as $n\to +\infty$. For the simplicity of the notations, we drop the subscript $n$.

Let
\[\eta_{\la}(x) = \frac{u_{\la}^{(1)}(x)-u_{\la}^{(2)}(x)}{\|u_{\la}^{(1)}-u_{\la}^{(2)}\|_{L^{\infty}(\mathbb{R}^{3})}}. \]
Then $\|\eta_{\la}\|_{L^{\infty}(\mathbb{R}^{3})}=1$ and $\eta_{\la}(x)$ satisfies the following equation
\[ 
-\Delta \eta_{\la}  + \la  \eta_{\la} 
+ \Big(|x|^{-1}\ast \big[\big(u_{\la}^{(1)}+u_{\la}^{(2)}\big)\eta_{\la}\big]\Big) u_{\la}^{(1)} 
+ \Big(|x|^{-1} \ast \big(u_{\la}^{(2)}\big)^2 \Big) \eta_{\la} = V(x) f(p_{\va}, u_{\la}^{(1)}, u_{\la}^{(2)} \big) \eta_{\la}, \]
where
\[f( p_{\va}, u_{\la}^{(1)}, u_{\la}^{(2)}) = ( p_{\va}-1)\int_{0}^{1} \big(tu_{\la}^{(1)} + (1-t)u_{\la}^{(2)}\big)^{p_{\va}-2} \dt.\]
We aim to get a contradiction by proving $\eta_{\la}=o(1)$ in $\r^3$.
Firstly we estimate $\eta_{\la}$ in $\mathbb{R}^{3}\setminus B_{\frac{R}{\sqrt{\la}}}(x^{(1)}_{\la})$.

\begin{lemma}\label{3-4} Fix $d>0$ small. There exist constants $C>0$ and $\tau>0$, such that
	\begin{equation}\label{3-4-1}
		|\eta_{\la}(x)| \leq C e^{-\tau \sqrt{\la}|x- x_{\la}^{(1)}|},  \text{ for any } x \in \mathbb{R}^{3}\setminus B_{\frac{R}{\sqrt{\la}}}(x^{(1)}_{\la}),
	\end{equation}
	and
	\begin{equation}\label{3-4-2}
		|\nabla \eta_{\la}(x)| \leq C e^{-\tau \sqrt{\la}},  \text{ for any } x \in  \partial B_{d}(x^{(1)}_{\la}).
	\end{equation}
\end{lemma}

\begin{proof}
	Note that for any small number $\tau>0$, there is  $R>0$ large, such that
	\[
	u_{\la}^{(j)}\le \la^{\frac{1}{p_\va-2}}\tau, \text{ in }  \mathbb{R}^{3}\setminus B_{\frac{R}{\sqrt{\la}}}(x^{(j)}_{\la}).
	\]
	From $|x^{(j)}_{\la} - b_{0}| = o(\la^{-\frac12})$, we can deduce
	\[
	u_{\la}^{(j)}\le \la^{\frac{1}{p_\va-2}}\tau, \text{ in }  \mathbb{R}^{3}\setminus B_{\frac{R}{\sqrt{\la}}}(x^{(1)}_{\la}).
	\]
	This gives that 
	\[
	\frac{1}{\la}V(x)f( p_{\va}, u_{\la}^{(1)}, u_{\la}^{(2)}) \le \frac14,  \text{ in }  \mathbb{R}^{3}\setminus B_{\frac{R}{\sqrt{\la}}}(x^{(1)}_{\la}),
	\]
	\[
	\frac{1}{\la} \Big(|x|^{-1} \ast (u_{\la}^{(2)})^2\Big) \leq \la^{\frac{2}{p_\va-2}-2} \int_{\mathbb{R}^{3}} \frac{1}{|x-y|}Q_{p_\va}^2(y) \dy \leq \frac14, \text{ for } \la \text{ large },
	\]
	and 
	\[
	\frac{1}{\la}
	\Big| \Big(|x|^{-1}\ast \big[\big(u_{\la}^{(1)}+u_{\la}^{(2)}\big)\eta_{\la}\big]\Big) u_{\la}^{(1)} \Big|
	\leq  e^{-\tau \sqrt{\la}|x- x_{\la}^{(1)}|} ,  \text{ in }  \mathbb{R}^{3}\setminus B_{\frac{R}{\sqrt{\la}}}(x^{(1)}_{\la}).
	\]
	Hence, by using the comparison principle,  we can prove \eqref{3-4-1} and \eqref{3-4-2}.
\end{proof}

Next we estimate $\eta_{\la}$ in $B_{\frac{R}{\sqrt{\la}}}(x^{(1)}_{\la})$.
We want to prove that $\eta_{\la} =o(1)$ in $B_{\frac{R}{\sqrt{\la}}}(x^{(1)}_{\la})$.  For this purpose, we define
\[\tilde{\eta}_{\la}(x) = \eta_{\la} \Big(\frac{x}{\sqrt{\la}} + x_{\la}^{(1)}\Big). \]
We have the following result.
\begin{lemma}\label{3-5} It holds that  as  $\la \to \infty$, $\va \to 0$,
	\[\tilde{\eta}_{\la} \to \sum_{j=1}^3 c_{j}\frac{\partial Q_{\bar{p}}}{\partial x_{j}}, \quad \text{in}\;  C^2_{loc}(\mathbb R^3) ,\]
	for some constants $c_{j}$, $j=1,2, 3$.
\end{lemma}
\begin{proof} It is easily verified that $\tilde{\eta}_{\la}$ satisfies
	\begin{equation}\label{3-5-1}
		\begin{split}
			&-\Delta \tilde{\eta}_{\la} +  \tilde{\eta}_{\la}
			+ \frac{1}{\la^2}  \Big[|x|^{-1}\ast \Big(u_{\la}^{(2)}\big(\frac{x}{\sqrt{\la}} + x_{\la}^{(1)}\big) \Big)^2\Big] \tilde{\eta}_\la \\
			&+ \frac{1}{\la^2}  \Big[|x|^{-1}\ast \Big(u_{\la}^{(1)}\big(\frac{x}{\sqrt{\la}} + x_{\la}^{(1)}\big) + u_{\la}^{(2)}\big(\frac{x}{\sqrt{\la}} + x_{\la}^{(1)}\big) \Big) \tilde{\eta}_\la \Big] u_{\la}^{(1)}\big(\frac{x}{\sqrt{\la}} + x_{\la}^{(1)}\big)  \\
			&=\frac{1}{\la} V\big(\frac{x}{\sqrt{\la}} + x_{\la}^{(1)}\big) f\Big(p_{\va}, u_{\la}^{(1)}\big(\frac{x}{\sqrt{\la}} + x_{\la}^{(1)}\big), u_{\la}^{(2)}\big(\frac{x}{\sqrt{\la}} + x_{\la}^{(1)}\big)\Big)\tilde{\eta}_{\la}, 
			\quad \text{ in } \mathbb{R}^{3}.
		\end{split}
	\end{equation}
	Then we obtain that
	\begin{equation}\label{3-5-2}
		\begin{split}
			-\Delta \tilde{\eta}_{\la} +  \tilde{\eta}_{\la} 
			= (p_{\va}-1) V_0 Q^{p_{\va}-2}_{p_{\va}} \tilde{\eta}_{\la} + o_{\la}(1) 
			= (\bar{p}-1) V_0 Q_{\bar{p}}^{\bar{p}-2} \tilde{\eta}_{\la} +o_\va(1) +o_\la(1),
			 \text{ in } \mathbb{R}^{3}.
		\end{split}
	\end{equation}
	Since $|\tilde{\eta}_{\la}|\le 1$,    $\tilde{\eta}_{\la} \to \tilde{\eta}$ in $C^{2}(B_{R}(0))$ for any $R>0$ and $\tilde{\eta}$ satisfies the equation
	\[-\Delta \tilde{\eta} + \tilde{\eta} =(\bar{p}-1) V_0 Q_{\bar{p}}^{\bar{p}-2} \tilde{\eta}, \text{ in  } \mathbb{R}^{3}. \]
	Hence we have
	\[ \tilde{\eta} = \sum_{j=1}^{3}c_{j} \frac{\partial Q_{\bar{p}}}{\partial x_{j}},\]
	by the non-degeneracy of $Q_{\bar{p}}$.
\end{proof}

\begin{lemma}\label{3-51}
	\[ \eta_{\la}(x) =o(1) \text{  in  } B_{\frac{R}{\sqrt{\la}}}(x^{(1)}_{\la}).\]
\end{lemma}
\begin{proof}
	By Lemma \ref{3-5}, we only need to prove $c_{j} =0$,  $j =1, 2, 3$.  Thanks to the symmetry, \eqref{Pohozaev} can be written as
	\begin{equation}\label{Pohozaev0}
		\begin{split}
			&\frac{1}{p_\va} \int_{ B_{d}(x_{\la})} \frac{\partial V(x)}{\partial x_j} u^{p_\va} \dx  \\ 
			=  {}& \int_{\partial B_{d}(x_{\la})} \frac{\partial u}{\partial \nu} \frac{\partial u}{\partial x_{j}} \dsi 
			-\frac12  \int_{\partial B_{d}(x_{\la})} |\nabla u|^2 \nu_j \dsi
			- \frac{\la }{2}\int_{\partial B_{d}(x_{\la})} u^2 \nu_j \dsi
			+ \frac{1}{p_\va} \int_{\partial B_{d}(x_{\la})} V(x) u^{p_\va} \nu_j \dsi  \\
			{}&- \frac12 \int_{\partial B_{d}(x_{\la})}\int_{\r^3} \frac{u^2(y)}{|x-y|} \dy \cdot u^2(x) \nu_j \dsi 
			+\frac12 \int_{\r^3\setminus B_{d}(x_{\la})}\int_{\r^3} \frac{x_j-y_j}{|x-y|^3} u^2(y) \dy \cdot u^2(x) \dx . 
		\end{split}
	\end{equation}
	Applying \eqref{Pohozaev0} to $u_{\la}^{(1)}$ and $u_{\la}^{(2)}$, we find
	\begin{equation}\label{p2}\begin{split}
			&\frac{1}{p_{\va} } \int_{ B_{d}(x^{(1)}_{\la})} \frac{\partial V(x)}{\partial x_j} F(p_{\va}, u_{\la}^{(1)}, u_{\la}^{(2)})  \eta_{\la}   \dx \\
			=  {}& \int_{\partial B_{d}(x^{(1)}_{\la})} \frac{\partial  u_{\la}^{(2)}}{\partial \nu} \frac{\partial \eta_{\la}}{\partial x_{j}} \dsi
			+ \int_{\partial B_{d}(x^{(1)}_{\la})} \frac{\partial  \eta_{\la}}{\partial \nu} \frac{\partial u_{\la}^{(1)}}{\partial x_{j}} \dsi
			- \frac{1}{2}  \int_{\partial B_{d}(x^{(1)}_{\la})} \nabla (u_{\la}^{(1)} + u_{\la}^{(2)}) \nabla \eta_{\la} \nu_{j}  \dsi \\
			{}&  - \frac{\la}2 \int_{\partial B_{d}(x^{(1)}_{\la})}  \big(u_{\la}^{(1)} + u_{\la}^{(2)}\big) \eta_{\la}  \nu_{j} \dsi 
			+ \frac{1}{p_{\va} } \int_{\partial B_{d}(x^{(1)}_{\la})} V(x) F\big(p_{\va}, u_{\la}^{(1)}, u_{\la}^{(2)}\big)  \eta_{\la}  \nu_{j} \dsi \\
			{}&-\frac12 \int_{\partial B_{d}(x^{(1)}_{\la})} \int_{\r^3} \frac{\big(u_{\la}^{(1)}(y)+u_{\la}^{(2)}(y) \big) \eta_{\la}(y)}{|x-y|} \dy \cdot  \big(u_{\la}^{(1)}(x) \big)^2 \nu_j \dsi \\
			{}&-\frac12 \int_{\partial B_{d}(x^{(1)}_{\la})} \int_{\r^3} \frac{\big(u_{\la}^{(2)}(y)\big)^2 }{|x-y|} \dy \cdot  \big(u_{\la}^{(1)}(x)+u_{\la}^{(2)}(x) \big) \eta_{\la}(x) \nu_j \dsi, \\
			{}&+\frac12 \int_{\r^3\setminus B_{d}(x^{(1)}_{\la})} \int_{\r^3} \frac{x_j-y_j}{|x-y|^3}\big(u_{\la}^{(1)}(y)+u_{\la}^{(2)}(y) \big) \eta_{\la}(y) \dy \cdot  \big(u_{\la}^{(1)}(x) \big)^2  \dx \\
			{}&+\frac12 \int_{\r^3\setminus B_{d}(x^{(1)}_{\la})} \int_{\r^3} \frac{x_j-y_j}{|x-y|^3} \big(u_{\la}^{(2)}(y)\big)^2\dy \cdot  \big(u_{\la}^{(1)}(x)+u_{\la}^{(2)}(x) \big) \eta_{\la}(x)  \dx \\
		\end{split}
	\end{equation}
	where
	\[ F\big(p_{\va}, u_{\la}^{(1)}, u_{\la}^{(2)}\big)= \frac{\big(u_{\la}^{(1)}\big)^{p_{\va}} - \big(u_{\la}^{(2)}\big)^{p_{\va}} }{u_{\la}^{(1)} - u^{(2)}_{\la}} = p_{\va} \int_{0}^{1}\big(tu_{\la}^{(1)} +(1-t)u_{\la}^{(2)}\big)^{p_{\va}-1}\dt.  \]
	It follows from Lemma~\ref{3-4}  that
	\begin{equation}\label{rhs}
		\text{RHS of \eqref{p2}} =  O(e^{-\theta_{2} \sqrt{\la}}),
	\end{equation}
	where $\theta_{2}> 0$ is a small constant. 
	By the Taylor's expansion and Lemma \ref{3-3}, we have
	\begin{equation}\label{lhs}
		\begin{split}
			\text{LHS of \eqref{p2}} ={}& \frac{1}{p_{\va} } \la^{-\frac32} \int_{ B_{\sqrt{\la}d}(x^{(1)}_{\la})} \frac{\partial V\big(\frac{x}{\sqrt{\la}}+x_{\la}^{(1)}\big)}{\partial x_j} F\Big(p_{\va}, u_{\la}^{(1)}\big(\frac{y}{\sqrt{\la}}+x_{\la}^{(1)}\big), u_{\la}^{(2)}\big(\frac{y}{\sqrt{\la}}+x_{\la}^{(1)}\big)\Big) \tilde{\eta}_{\la}   \dx \\
			={}&  \frac{1}{p_{\va}} \la^{-\frac32} \frac{\partial V(x_{\la}^{(1)})}{\partial x_j}  \int_{ B_{d}(x^{(1)}_{\la})} F\Big(p_{\va}, u_{\la}^{(1)}\big(\frac{y}{\sqrt{\la}}+x_{\la}^{(1)}\big), u_{\la}^{(2)}\big(\frac{y}{\sqrt{\la}}+x_{\la}^{(1)}\big)\Big) \tilde{\eta}_{\la}   \dx \\
			{}& + \frac{1}{p_{\va}} \la^{-2}  \sum_{h=1}^3 \int_{ B_{d}(x^{(1)}_{\la})} \frac{\partial^2 V(x_{\la}^{(1)})}{\partial x_j\partial x_h} x_h F\Big(p_{\va}, u_{\la}^{(1)}\big(\frac{y}{\sqrt{\la}}+x_{\la}^{(1)}\big), u_{\la}^{(2)}\big(\frac{y}{\sqrt{\la}}+x_{\la}^{(1)}\big)\Big) \tilde{\eta}_{\la}   \dx \\
			{}& + O(\la^{\frac{1}{p_\va-2}-\frac32}) \\
			={}& \frac{1}{p_{\va}} \la^{-2}  \sum_{h=1}^3 \int_{ B_{d}(x^{(1)}_{\la})} \frac{\partial^2 V(x_{\la}^{(1)})}{\partial x_j\partial x_h} x_h F\Big(p_{\va}, u_{\la}^{(1)}\big(\frac{y}{\sqrt{\la}}+x_{\la}^{(1)}\big), u_{\la}^{(2)}\big(\frac{y}{\sqrt{\la}}+x_{\la}^{(1)}\big)\Big) \tilde{\eta}_{\la}   \dx \\
			{}&+o(\la^{\frac{1}{p_\va-2}-1}).\\
		\end{split}
	\end{equation}
	Combining \eqref{rhs} and \eqref{lhs} yields
	\begin{equation}\label{phozaev}
		\frac{1}{p_{\va}} \la^{-2}  \sum_{h=1}^3 \int_{ B_{d}(x^{(1)}_{\la})} \frac{\partial^2 V(x_{\la}^{(1)})}{\partial x_j\partial x_h} x_h F\Big(p_{\va}, u_{\la}^{(1)}\big(\frac{y}{\sqrt{\la}}+x_{\la}^{(1)}\big), u_{\la}^{(2)}\big(\frac{y}{\sqrt{\la}}+x_{\la}^{(1)}\big)\Big) \tilde{\eta}_{\la}   \dx
		=o(\la^{\frac{1}{p_\va-2}-1}).
	\end{equation}
	Taking the limit in \eqref{phozaev}, we obtain from Lemma \ref{3-5} that
	\begin{equation}
		\sum_{h=1}^3 \int_{\r^3} \frac{\partial^2 V(b_0)}{\partial x_j\partial x_h} x_h Q_{\bar{p}}^{\bar{p}-1} \sum_{l=1}^3c_l \frac{\partial Q_{\bar{p}}}{\partial x_l} \dx=0.
	\end{equation}
	By the symmetry, we have
	\begin{equation}
		\sum_{h=1}^3 c_h \frac{\partial^2 V(b_0)}{\partial x_j\partial x_h} \int_{\r^3} x_h Q_{\bar{p}}^{\bar{p}-1}  \frac{\partial Q_{\bar{p}}}{\partial x_l} \dx =0.
	\end{equation}
	Since $b_0$ is a non-degenerate critical point of $V(x)$ and 
	\begin{equation*}
		\int_{\r^3}x_h Q_{\bar{p}}^{\bar{p}-1}  \frac{\partial Q_{\bar{p}}}{\partial x_l}  \dx
		=-\frac{1}{\bar{p}} \int_{\r^3} Q_{\bar{p}}^{\bar{p}} \dx <0,
	\end{equation*}
	we obtain $c_h=0$, for $h=1,2,3$.
\end{proof}

\begin{proof} [\bf Proof of Theorem~\ref{3-7}.]
	We accomplish the proof by Lemmas~\ref{3-4} and \ref{3-51}.
\end{proof}

\section{Normalized  solutions}

\subsection{Existence of normalized  solutions}

In this subsection, we will prove Theorem \ref{ExistenceOfNormalized}.

In Section~2, we prove that there are $\va_0>0$ and $\la_0>0$, such that
for any $\va\in (0, \va_0)$ and $\la>\la_0$, the problem
\[ -\Delta u + \lambda u +  \big(|x|^{-1} \ast |u|^{2}\big)u= V(x) u^{p_{\varepsilon}-1 } ,  \text{ in }  \r^{3},\]
has a unique solution
of the form
\begin{equation}\label{u-la}
	u_{\la} = U_{x_{\la},p_\va} + \omega_\la,
\end{equation}
where  $U_{x_{\la},p_\va}(x)=\big(\frac{\la}{V_0}\big)^{\frac{1}{p_{\va}-2}}Q_{p_\va} \big(\sqrt{\lambda} (x - x_{\la})\big),$ and $\omega_\la\in E_{\lambda ,p_\va} $. Moreover, $x_{\la}$ satisfies $|x_{\la}-b_0| = o(\la^{-\frac12})$ and
$\|\omega_{\la}\|_{\la}\leq C  \la^{\frac{3}{p_{\va}-2} -\frac{9}{4}}.$

Define 
\[\bar u_{\la} := \frac{\sqrt{a} u_{\la}}{\|u_{\la}\|_{L^{2}(\mathbb{R}^{3})}}
\,\text{ and } \,
f(\la) := \frac{ \| u_{\la}\|_{L^{2}(\mathbb{R}^{3})}^2}{a}.
\]
Then $\bar u_{\la}$ satisfies the equation
\begin{equation}\label{4-1-1}\left\{\begin{array}{lll}
-\Delta \bar u_{\la} + \lambda \bar u_{\la} 
+ f(\la) \big(|x|^{-1} \ast |\bar u_{\la}|^{2}\big)\bar u_{\la}
= V(x) \big(f(\la)\big)^{\frac{p_{\va}}{2}-1} \bar{u}_{\la}^{p_{\va} -1} , \text{ in } \r^{3},\\[2mm]
\int_{\r^3}|\bar u_{\la}|^{2}dx =a. 
\end{array}
\right.
\end{equation}

\begin{remark}
	Theorem \ref{3-7} is necessary to prove the existence of normalized. We aim to apply the intermediate value theorem to prove $f(\la)=1$ for some $\la$, which force that the continuity of $f(\la)$ comes first, so is the uniqueness of the solution $u_\la$ by the definition of $f(\la)$. In fact, for any $\tilde{\la}\in(\lambda_0, +\infty)$, 
	\[
	\lim_{\la\to\tilde{\la}} f(\la) = \frac{1}{a} \lim_{\la\to\tilde{\la}} \| u_{\la}\|_{L^{2}(\mathbb{R}^{3})}^2  = \frac{1}{a} \| u_{\tilde{\la}}\|_{L^{2}(\mathbb{R}^{3})}^2 = f(\tilde{\la}),
	\]
	thanks to the uniqueness of the solution $u_\la$. So $f(\la)$ is a continuous function in $(\lambda_0, +\infty)$.
\end{remark}

\begin{proof}[\bf Proof of Theorem \ref{ExistenceOfNormalized}.]
Our purpose is to prove $f(\la) =1 $ for some large $\lambda$. 

By calculating directly, we have
\begin{equation}\label{3-3-11}\begin{split}
f(\la)= {}&\frac{1}{a} \int_{\mathbb{R}^{3}} \bigl( U_{x_{\lambda},p_\va} +\omega_{\la} \bigr)^{2} \dx   \\
= {}&\frac{1}{a}\int_{\mathbb{R}^{3}} U_{x_{\la},p_\va} ^{2} \dx + O\Bigl( \|U_{x_{\la},p_\va}\|_{L^{2}(\mathbb{R}^{3})} \|\omega_{\la}\|_{L^{2}(\mathbb{R}^{3})} + \|\omega_{\la}\|^{2}_{L^{2}(\mathbb{R}^{3})}  \Bigr)\\
= {}& \frac{1}{a} \int_{\mathbb{R}^{3}} Q_{p_{\va}} ^{2}(x) \dx \cdot \la^{\frac{2}{p_{\va}-2}-\frac{3}{2}}  V_0^{-\frac{2}{p_{\va}-2}}  +  O\Bigl( \la^{\frac1{p_\va-2}-\frac{3}{4}} \|\omega_{\la}\|_{L^{2}(\mathbb{R}^{3})} + \|\omega_{\la}\|^{2}_{L^{2}(\mathbb{R}^{3})} \Bigr).
\end{split}
\end{equation}
By Lemma \ref{4-1},
 we have
\[
\| Q_{p_{\va}}(x) - Q_{\bar{p}}(x)\|_{L^2(\mathbb R^3)} \leq C \va
. \]
By Proposition \ref{2-4},
\[
\|\omega_{\la}\|_{L^{2}(\mathbb{R}^{3})} 
\le C\la^{-\frac12} \|\omega_{\la}\|_\la
\le C \la^{\frac{3}{p_\va-2} -\frac{11}{4}}.
\]
Then it follows that
\[
f(\la) = \frac{a_{*}}{V_0^{\frac32}a} \lambda^{\frac{2}{p_{\va}-2}-\frac{3}{2}}
\Bigl( 1+O(\va) \Bigr)
+  O\Bigl( \la^{\frac4{p_\va-2}-\frac{7}{2}} \Bigr). \]

In Case $(i)$, we have $  a<V_0^{-\frac32} a_{*}$ and $p_\va = \frac{10}{3} +\va$.
Hence,
\[\begin{split}
f\Big( e^{\frac{4}{9\va} \ln (V_0^{-\frac32} \frac{a_{*}}{a} ) }\Big) ={}& \biggl( \frac{a_{*}}{V_0^{\frac32}a} \biggr)^{1- \frac{2}{4+3\va} }
\Bigl( 1+O(\va) \Bigr) 
+  O\Bigl( \la^{\frac4{p_\va-2}-\frac{7}{2}} \Bigr) \\
={}& \biggl( \frac{a_{*}}{V_0^{\frac32}a} \biggr)^{\frac{1}{2} + O(\va)}
\Bigl( 1+O(\va) \Bigr) 
+  O\Bigl( \la^{\frac4{p_\va-2}-\frac{7}{2}} \Bigr)  > 1
\end{split}
  \]
and
\[\begin{split}
f\Big( e^{\frac{16}{9\va} \ln (V_0^{-\frac32}\frac{a_{*}}{a})}\Big) ={}& \biggl( \frac{a_{*}}{V_0^{\frac32}a} \biggr)^{1- \frac{8}{4+3\va} }
\Bigl( 1+O(\va) \Bigr) 
+  O\Bigl( \la^{\frac4{p_\va-2}-\frac{7}{2}} \Bigr) \\
={}& \biggl( \frac{a_{*}}{V_0^{\frac32}a} \biggr)^{-1+ O(\va)}
\Bigl( 1+O(\va) \Bigr) 
+  O\Bigl( \la^{\frac4{p_\va-2}-\frac{7}{2}} \Bigr)  < 1 .
\end{split}
  \]
Since $f(\la)$ is continuous in $(\lambda_0, +\infty)$, the intermediate value theorem shows that there exists $\la_{\va} \in \Big(e^{\frac{4}{9\va} \ln (V_0^{-\frac32}\frac{a_{*}}{a})}, e^{\frac{16}{9\va} \ln (V_0^{-\frac32}\frac{a_{*}}{a})}\Big)$, such that $f(\la_{\va })=1$. 

The result for Case $(ii)$ can be argued in a similar way.

Up to this point, we demonstrate the dependence of $\la$ and $\va$. Let us point out that $\lambda_\va\to +\infty$ as $\va\to 0$.  Thus problem \eqref{maineq} has a single-peak solution $u_\va$, such that as $\va\to0$,
$u_\va$ blows up at a critical point $b_0$ of $V(x)$. And the solution of form \eqref{u-la} can be rewritten as
\begin{equation}\label{60-21-5}
	u_{\va} =  U_{x_{\va},p_\va} + \omega_\va,
\end{equation}
 where \[U_{x_{\va},p_\va}(x)=\Big(\frac{\la_\va}{V_0}\Big)^{\frac{1}{p_{\va}-2}}Q_{p_\va} \big(\sqrt{\lambda_\va} (x - x_{\va})\big),\]
 and $ \omega_\va\in E_{\va}$, where
 \begin{equation}\label{2-e1}
 	E_{\va} := \biggl\{u \in H^{1}(\mathbb{R}^{3}): \langle u, \frac{\partial U_{x_{\va}, p_{\va}}}{\partial x_{j}}\rangle_{\la_\va}=0, j=1,2,3 \biggr\}.
 \end{equation}
We accomplish the proof of Theorem \ref{ExistenceOfNormalized}.
\end{proof}

\subsection{Local uniqueness of normalized solutions}

In last subsection, we have proven that for $\va>0$ small, problem \eqref{maineq} has a single-peak solution $u_\va$ of form \eqref{60-21-5}.

Lemmas \ref{2-2} and \ref{3-3} show that
\begin{equation}\label{62-21-5}
|x_{\va} - b_0| = o_\va(\la_\va^{-\frac12}),\quad \|\omega_{\va}\|_{\la_\va}\leq  C  \la_{\va}^{\frac{3}{p_{\va}-2} -\frac{9}{4}}.
\end{equation}

In this subsection, we will prove the local uniqueness of $u_\va$ of the form \eqref{60-21-5}, satisfying \eqref{62-21-5}.
Firstly, we clarify the dependence of $\la_\va$ and $\va$ more precisely.
\begin{lemma}\label{prop-lambda}
As $\va \to 0$,  it holds
\begin{equation}\label{lam1}
\lambda_{\varepsilon}=\Lambda_\va\bigg(1+O\Big(\frac{\ln \Lambda_\va}{\Lambda_\va^{\frac{7}{2}-\frac{4}{p_\va-2}}}\Big)\bigg),
\end{equation}
where
 \[\Lambda_\va:=\biggl(V_0^{\frac{2}{p_\va-2}}\frac{a}{a_{*,\va}} \bigg)^{\frac{2(p_{\va}-2)}{10-3p_\va}},\]
and
$$a_{*,\va}=\int_{\mathbb{R}^3}Q_{p_\va}^2\dx.$$
\end{lemma}

\begin{proof}
Similar to \eqref{3-3-11}, we compute
\[
\begin{split}
\int_{\mathbb{R}^{3}} u_{\varepsilon}^{2} \dx
= {}&\int_{\mathbb{R}^{3}} U_{x_{\va},p_\va} ^{2} \dx + O\Bigl( \|U_{x_{\va},p_\va}\|_{L^{2}(\mathbb{R}^{3})} \|\omega_{\la}\|_{L^{2}(\mathbb{R}^{3})} + \|\omega_{\la}\|^{2}_{L^{2}(\mathbb{R}^{3})}  \Bigr)\\
= {}& \int_{\mathbb{R}^{3}} Q_{p_{\va}} ^{2}(x) \dx \cdot \la_\va^{\frac{2}{p_{\va}-2}-\frac{3}{2}}  V_0^{-\frac{2}{p_{\va}-2}}  
+  O\Bigl( \la_\va^{\frac4{p_\va-2}-\frac{7}{2}} \Bigr).
\end{split}
\]
Since $ \int_{\mathbb{R}^{3}} u_{\varepsilon}^{2} \dx=a$, we obtain
\begin{equation*}
\la_{\va}^{\frac{10-3p_\va}{2(p_{\va}-2)}} = V_0^{\frac{2}{p_\va-2}}\frac{a}{a_{*,\va}}  + O\Bigl( \la_\va^{\frac4{p_\va-2}-\frac{7}{2}} \Bigr).
\end{equation*}
Let \[\Lambda_\va:=\bigg(V_0^{\frac{2}{p_\va-2}}\frac{a}{a_{*,\va}} \bigg)^{\frac{2(p_{\va}-2)}{10-3p_\va}},\] we have
\[
\begin{aligned}
\la_{\va}=\Lambda_\va \Bigl[1+ O\Big( \la_\va^{\frac4{p_\va-2}-\frac{7}{2}} \Big) \Bigr]^{\frac{2(p_{\va}-2)}{10-3p_\va}}
=\Lambda_\va e^{\frac{2(p_{\va}-2)}{10- 3p_{\va}}\ln \Big(1+O\big(  \la_\va^{\frac4{p_\va-2}-\frac{7}{2}} \big)\Big)}
=\Lambda_\va \bigg(1+O\Big(\va^{-1} \la_\va^{\frac{4}{p_\va-2}-\frac{7}{2}}\Big)\bigg).
\end{aligned}
\]
Therefore,
\[\lambda_{\varepsilon}=\Lambda_\va\bigg(1+O\Big(\frac{\ln \Lambda_\va}{\Lambda_\va^{\frac{7}{2}-\frac{4}{p_\va-2} }}\Big)\bigg).\]
\end{proof}

Lemma \ref{prop-lambda} implies $\la_\va^\va$ tends to some positive constant $c_0$ as $\va\to0$. Revisit the proof of Lemma \ref{3-3}, then the equation \eqref{3-3-6} indicates that 
\begin{equation}\label{lambda4}
	|x_{\va} - b_0| = O(\la_\va^{-1}).
\end{equation}

 To prove the local uniqueness of the peak solution, we proceed by contradiction. Suppose $u_{\varepsilon}^{(1)}$ and $u_{\varepsilon}^{(2)}$ are
  two different solutions of \eqref{maineq}
  of the form \eqref{60-21-5}. We define
$$
\xi_{\varepsilon}=\frac{u_{\varepsilon}^{(1)}-u_{\varepsilon}^{(2)}}{\|u_{\varepsilon}^{(1)}-u_{\varepsilon}^{(2)}\|_{L^{\infty}(\mathbb{R}^{3})}},
$$
then $\xi_{\varepsilon}$ satisfies $\|\xi_{\varepsilon}\|_{L^{\infty}(\r^3)}=1$ and
\[
\begin{split}
	{}&-\Delta \xi_{\varepsilon}  + \la_\va^{(1)}  \xi_{\varepsilon} 
	+ \Big(|x|^{-1}\ast \big[\big(u_{\va}^{(1)}+u_{\va}^{(2)}\big)\xi_{\varepsilon}\big]\Big) u_{\va}^{(1)} 
	+ \Big(|x|^{-1} \ast \big(u_{\va}^{(2)}\big)^2\Big) \xi_{\varepsilon} \\
	={}& \frac{V(x)\Big[\big(u_{\varepsilon}^{(1)}\big)^{p_{\varepsilon}-1}-\big(u_{\varepsilon}^{(2)}\big)^{p_{\varepsilon}-1}\Big]-\big(\lambda_{\varepsilon}^{(1)}-\lambda_{\varepsilon}^{(2)}\big) u_{\varepsilon}^{(2)}}{\|u_{\varepsilon}^{(1)}-u_{\varepsilon}^{(2)}\|_{L^{\infty}(\mathbb{R}^{3})}}.
\end{split}
\]

Define
$$
C_{\varepsilon}(x):=\lambda_{\varepsilon}^{(1)} 
+ \Big(|x|^{-1} \ast \big(u_{\va}^{(2)}\big)^2\Big) \xi_{\varepsilon}  
-( p_{\varepsilon} -1) V(x) \int_{0}^{1}\big(t u_{\varepsilon}^{(1)}+(1-t) u_{\varepsilon}^{(2)}\big)^{p_{\varepsilon}-2} \dt ,
$$
and
$$
g_{\va}(x):=\frac{\lambda_{\varepsilon}^{(2)}-\lambda_{\varepsilon}^{(1)}}{\|u_{\varepsilon}^{(1)}-u_{\varepsilon}^{(2)}\|_{L^{\infty}(\mathbb{R}^{3})}} u_{\varepsilon}^{(2)}
- \Big(|x|^{-1}\ast \big[\big(u_{\va}^{(1)}+u_{\va}^{(2)}\big)\xi_{\varepsilon}\big]\Big) u_{\va}^{(1)} .
$$
Then
$$
-\Delta \xi_{\va}(x)+C_{\va}(x) \xi_{\va}(x) = g_{\va}(x).
$$

Let
\begin{equation}\label{0-23-5}
\bar{\xi}_{\varepsilon}(x)=\xi_{\varepsilon}\Big(\frac{x}
{\sqrt{\lambda_{\varepsilon}^{(1)}}}+x_{\varepsilon}^{(1)}\Big),
\end{equation}
then
\begin{equation}\label{1-23-5}
-\Delta \bar{\xi}_{\varepsilon}+\frac{1}{\lambda_{\varepsilon}^{(1)}} C_{\varepsilon}\Big(\frac{x}{\sqrt{\lambda_{\varepsilon}^{(1)}}}+x_{\varepsilon}^{(1)}\Big) \bar{\xi}_{\varepsilon}=\frac{1}{\lambda_{\varepsilon}^{(1)}} g_{\varepsilon}\Big(\frac{x}{\sqrt{\lambda_{\varepsilon}^{(1)}}}+x_{\varepsilon}^{(1)}\Big).
\end{equation}

We first find the main terms of $C_\va$ and $g_\va$ in \eqref{1-23-5}.

\begin{lemma}\label{blowup}
Let $d>0$ be a small constant.
For any $x\in B_{\sqrt{\la_\va^{(1)}}d}(0)$, there hold
\begin{equation}\label{5-4-5}
\begin{split}
\frac{1}{\la_{\va}^{(1)}} C_{\va}\Big(\frac{x}{\sqrt{\la_{\va}^{(1)}}}+x_{\va}^{(1)}\Big)
=& 1 - (p_\va-1)  Q_{p_\va}^{p_\va-2}(x)
+O\biggl(\frac{\ln \la_{\va}^{(1)}}{\sqrt{\la_{\va}^{(1)}}}\biggr),
\end{split}
\end{equation}
and
\begin{equation}\label{5-4-6}
\begin{split}
\frac{1}{\la_{\va}^{(1)} } g_{\va}\Big(\frac{x}{\sqrt{\lambda_{\va}^{(1)}}}+x_{\va}^{(1)}\Big)
=&-\frac{p_\va-2}{a_{*,\va}}Q_{p_\va}(x)  \int_{\mathbb{R}^{3}} Q_{p_{\va}}^{p_{\va}-1} \bar{\xi}_{\va} \dx 
+O\biggl(\frac{\ln \la_\va^{(1)}}{\sqrt{\la_\va^{(1)}}}Q_{p_{\va}}(x)
\biggr).
\end{split}
\end{equation}
\end{lemma}

\begin{proof}
Lemma~\ref{prop-lambda} and estimate \eqref{lambda4} imply that
\begin{equation}\label{5-4-3}
 \frac{\la_\va^{(2)}}{\la_\va^{(1)}}=1+O\bigg(\frac{\ln \la_\va^{(1)}}{\big(\la_\va^{(1)}\big)^{\frac{7}{2}-\frac{4}{p_\va-2}}}\bigg)
\end{equation}
 and
\begin{equation}\label{5-4-4}
| x_{\va}^{(1)}-x_{\va}^{(2)}| = O\Bigl((\la_\va^{(1)})^{-1}\Bigr) .
\end{equation}
Hence, we have, for $x \in B_{\sqrt{\lambda_{\va}^{(1)}}d }(0)$,
\begin{equation*}
\begin{split}
&u_{\va}^{(1)}\Big(\frac{x}{\sqrt{\la_{\va}^{(1)}}}+x_{\va}^{(1)}\Big)\\
={}& \Big(\frac{\la_{\va}^{(1)}}{V_0}\Big)^{\frac{1}{p_{\va}-2}} Q_{p_{\va}}\bigg(\sqrt{\la_{\va}^{(1)}} \Big(\frac{x}{\sqrt{\la_{\va}^{(1)}}}+x_{\va}^{(1)} -x_{\va}^{(1)} \Big) \bigg) +\omega_{\va}^{(1)}\Big(\frac{x}{\sqrt{\la_{\va}^{(1)}}}+x_{\va}^{(1)}\Big)  \\
={}&\Big(\frac{\la_{\va}^{(1)}}{V_0}\Big)^{\frac{1}{p_{\va}-2}}  Q_{p_{\va}}(x)
+\omega_{\va}^{(1)}\Big(\frac{x}{\sqrt{\la_{\va}^{(1)}}}+x_{\va,i}^{(1)}\Big) ,
\end{split}
\end{equation*}
and
\begin{equation}\label{1-11-6}
\begin{split}
&u_{\va}^{(2)}\Big(\frac{x}{\sqrt{\la_{\va}^{(1)}}}+x_{\va}^{(1)}\Big)\\
={}& \Big(\frac{\la_{\va}^{(2)}}{V_0}\Big)^{\frac{1}{p_{\va}-2}}  Q_{p_{\va}}\bigg(\sqrt{\la_{\va}^{(2)}} \Big(\frac{x}{\sqrt{\la_{\va}^{(1)}}}+x_{\va}^{(1)} -x_{\va}^{(2)} \Big) \bigg) +\omega_{\va}^{(2)}\Big(\frac{x}{\sqrt{\la_{\va}^{(1)}}}+x_{\va}^{(1)}\Big) \\
={}& \Big(\frac{\la_{\va}^{(1)}}{V_0}\Big)^{\frac{1}{p_{\va}-2}} Q_{p_{\va}} (x)  \bigg[1 + O\bigg(\frac{\ln \la_\va^{(1)}}{\big(\la_\va^{(1)}\big)^{\frac{7}{2}-\frac{4}{p_\va-2}}}\bigg)  \bigg] 
+ O\biggl(\frac{\ln \la_\va^{(1)}}{\big(\la_\va^{(1)}\big)^{\frac{7}{2}-\frac{5}{p_\va-2}}} |\nabla Q_{p_{\va}}|\biggr) +\omega_{\va}^{(2)}\Big(\frac{x}{\sqrt{\la_{\va}^{(1)}}}+x_{\va}^{(1)}\Big).
\end{split}
\end{equation}
Thus,
\begin{equation}\label{sum-u}
	\begin{split}
		& t u_{\varepsilon}^{(1)}\Big(\frac{x}{\sqrt{\la_{\va}^{(1)}}}+x_{\va}^{(1)}\Big)
		+(1-t) u_{\varepsilon}^{(2)}\Big(\frac{x}{\sqrt{\la_{\va}^{(1)}}}+x_{\va}^{(1)}\Big) \\
		={}&  \Big(\frac{\la_{\va}^{(1)}}{V_0}\Big)^{\frac{1}{p_{\va}-2}} Q_{p_{\va}} (x)  \bigg[1 + O\bigg(\frac{\ln \la_\va^{(1)}}{\big(\la_\va^{(1)}\big)^{\frac{7}{2}-\frac{4}{p_\va-2}}}\bigg)  \bigg] 
		+ O\bigg(\frac{\ln \la_\va^{(1)}}{\big(\la_\va^{(1)}\big)^{\frac{7}{2}-\frac{5}{p_\va-2}}} |\nabla Q_{p_{\va}}|\bigg) + O\bigg(\sum_{l=1}^2 \omega_{\va}^{(l)}\Big(\frac{x}{\sqrt{\la_{\va}^{(1)}}}+x_{\va}^{(1)}\Big) \bigg) .\\
	\end{split}
\end{equation}
We deduce that
\begin{eqnarray*}
&&\frac{1}{\la_{\va}^{(1)}} C_{\va}\Big(\frac{x}{\sqrt{\la_{\va}^{(1)}}}+x_{\va}^{(1)}\Big)\\
&=&1
+\frac{1}{\la_{\va}^{(1)}} \bigg(\Big|\frac{x}{\sqrt{\la_{\va}^{(1)}}}+x_{\va}^{(1)}\Big|^{-1} \ast (u_{\va}^{(2)})^2\bigg) \bar \xi_{\varepsilon} \\
&&-\frac{1}{\la_{\va}^{(1)}} (p_{\va}-1) V\Big(\frac{x}{\sqrt{\la_{\va}^{(1)}}}+x_{\va}^{(1)}\Big) 
\int_{0}^{1}\bigg(t u_{\varepsilon}^{(1)}\Big(\frac{x}{\sqrt{\la_{\va}^{(1)}}}+x_{\va}^{(1)}\Big)
+(1-t) u_{\varepsilon}^{(2)}\Big(\frac{x}{\sqrt{\la_{\va}^{(1)}}}+x_{\va}^{(1)}\Big)\bigg)^{p_{\varepsilon}-2} \dt\\
&=& 1 -  (p_\va-1) \frac{V\Big(\frac{x}{\sqrt{\la_{\va}^{(1)}}}+x_{\va}^{(1)}\Big) }{V_0} Q^{p_\va-2}_{p_{\va}}(x)  \bigg\{   1 + O\bigg(\frac{\ln \la_\va^{(1)}}{\sqrt{\la_\va^{(1)}}} \bigg)  \nonumber\\
&& + O\bigg( \frac{1}{\big(\la_{\va}^{(1)}\big)^{\frac{1}{p_{\va}-2}}Q_{p_\va}(x)}\sum_{l=1}^2 \omega_{\va}^{(l)}\Big(\frac{x}{\sqrt{\la_{\va}^{(1)}}}+x_{\va}^{(1)}\Big)
\bigg)   \bigg\} ^{p_{\va}-2} 
+O\bigg(\frac{1}{\sqrt{\la_{\va}^{(1)}}}\bigg) \\
&=& 1 - (p_\va-1) Q_{p_\va}^{p_\va-2}(x)
+O\bigg(\frac{\ln \la_\va^{(1)}}{\sqrt{\la_{\va}^{(1)}}}\bigg),
\end{eqnarray*}
 since, by Lemma~\ref{3-1},
\[
 Q^{p_\va-3}_{p_{\va}}(x)\Big|\omega_{\va}^{(l)}
 \Big(\frac{x}{\sqrt{\la_{\va}^{(1)}}}+x_{\va}^{(1)}\Big)\Big|\le C.
 \]

Next, we estimate $g_\va$.
From \eqref{maineq}, we find
$$
\int_{\mathbb{R}^{3}} |\nabla u_{\va}^{(l)}|^{2} \dx +a \lambda_{\va}^{(l)} + \int_{\r^3} \int_{\r^3} \frac{1}{|x-y|} \big(u_{\va}^{(l)}(x)\big)^2 \big(u_{\va}^{(l)}(y)\big)^2 \dx\dy 
= \int_{\mathbb{R}^{3}}V(x)\big(u_{\va}^{(l)}\big)^{p_{\va}} \dx ,
$$
which gives
$$
\begin{aligned}
\frac{a \big(\lambda_{\va}^{(2)}-\lambda_{\va}^{(1)}\big)}{\|u_{\va}^{(1)}-u_{\va}^{(2)}\|_{L^{\infty}(\mathbb R^{3})}}
={}& \int_{\mathbb{R}^{3}}  \nabla\big(u_{\va}^{(1)}+u_{\va}^{(2)}\big) \cdot \nabla  \xi_{\va} \dx
+  \int_{\mathbb{R}^{3}} \frac{V(x)\big[\big(u_{\va}^{(2)}\big)^{p_{\varepsilon}}
-\big(u_{\va}^{(1)}\big)^{p_{\varepsilon}}\big]}{\|u_{\va}^{(1)}
-u_{\va}^{(2)}\|_{L^{\infty}(\mathbb{R}^{3})}} \dx \\
{}&+  \int_{\r^3} \int_{\r^3} \frac{1}{|x-y|} \big[u_{\va}^{(1)}(x)+u_{\va}^{(2)}(x)\big] \xi_{\va}(x) \big(u_{\va}^{(1)}(y)\big)^2 \dx\dy \\
{}& + \int_{\r^3} \int_{\r^3} \frac{1}{|x-y|} \big(u_{\va}^{(2)}(x)\big)^2 \big[u_{\va}^{(1)}(y)+u_{\va}^{(2)}(y)\big] \xi_{\va}(y) \dx\dy .
\end{aligned}
$$

Since
$$
\int_{\mathbb{R}^{3}}\big(u_{\va}^{(1)}+u_{\va}^{(2)}\big) \xi_{\varepsilon} \dx =\frac{1}{\|u_{\va}^{(1)}-u_{\va}^{(2)}\|_{L^{\infty}(\mathbb{R}^3)}} \bigg[\int_{\mathbb{R}^{3}}\big(u_{\va}^{(1)}\big)^{2} \dx -\int_{\mathbb{R}^{3}}\big(u_{\va}^{(2)}\big)^{2} \dx\bigg]=0,
$$
and
\begin{equation}\label{decompose}
\begin{aligned}
  &\big(u_\va^{(1)}\big)^{p_\va}-\big(u_\va^{(2)}\big)^{p_\va}\\
= {}&\big(u_\va^{(1)}\big)^{p_\va}-\big(u_\va^{(1)}\big)^{p_\va-1}u_\va^{(2)}+\big(u_\va^{(1)}\big)^{p_\va-1}u_\va^{(2)}
   -\Bigl[\big(u_\va^{(2)}\big)^{p_\va}-\big(u_\va^{(2)}\big)^{p_\va-1}u_\va^{(1)}+\big(u_\va^{(2)}\big)^{p_\va-1}u_\va^{(1)}\Bigr]\\
= {}&\big(u_\va^{(1)}\big)^{p_\va-1}\big(u_\va^{(1)}-u_\va^{(2)}\big) + \big(u_\va^{(2)}\big)^{p_\va-1}\big(u_\va^{(1)}-u_\va^{(2)}\big)
   +u_\va^{(1)}u_\va^{(2)}\Bigl[\big(u_\va^{(1)}\big)^{p_\va-2}-\big(u_\va^{(2)}\big)^{p_\va-2}\Bigr],
\end{aligned}
\end{equation}
we have
\begin{equation}\label{dec1}
  \begin{split}
    &\frac{a \big(\lambda_{\va}^{(2)}-\lambda_{\va}^{(1)}\big)}{\|u_{\va}^{(1)}-u_{\va}^{(2)}\|_{L^{\infty}(\mathbb{R}^3)}} \\
    ={}&  \lambda_{\va}^{(2)} \int_{\mathbb{R}^{3}}\big(u_{\va}^{(1)}+u_{\va}^{(2)}\big) \xi_{\va} \dx + \int_{\mathbb{R}^{3}}  \nabla\big(u_{\va}^{(1)}+u_{\va}^{(2)}\big) \cdot \nabla \xi_{\va} \dx \\
     {}& -  \int_{\mathbb{R}^{3}} V(x)\Bigg[ \big(u_\va^{(1)}\big)^{p_\va-1}\xi_{\va} + \big(u_\va^{(2)}\big)^{p_\va-1} \xi_{\va}+ \frac{u_\va^{(1)}u_\va^{(2)}\big((u_\va^{(1)})^{p_\va-2}-(u_\va^{(2)})^{p_\va-2}\big)}{\|u_{\va}^{(1)}-u_{\va}^{(2)}\|_{L^{\infty}(\mathbb{R}^3)}} \Bigg] \dx 
     +O\Big(\big(\lambda_{\va}^{(1)}\big)^{-\frac14} \Big) \\
    ={}& \big(\lambda_{\va}^{(2)}-\lambda_{\va}^{(1)}\big) \int_{\mathbb{R}^{3}} u_{\va}^{(1)} \xi_{\va} \dx
    -\int_{\mathbb{R}^{3}} \frac{V(x) u_\va^{(1)}u_\va^{(2)}\Bigl[\big(u_\va^{(1)}\big)^{p_\va-2}-\big(u_\va^{(2)}\big)^{p_\va-2}\Bigr]}{\|u_{\va}^{(1)}-u_{\va}^{(2)}\|_{L^{\infty}(\mathbb{R}^3)}} \dx
    +O\Big(\big(\lambda_{\va}^{(1)}\big)^{-\frac14} \Big) \\
    ={}&  \big(\lambda_{\va}^{(2)}-\lambda_{\va}^{(1)}\big) \int_{\mathbb{R}^{3}} u_{\va}^{(1)} \xi_{\va} \dx 
    -(p_{\va}-2) \int_{\mathbb{R}^{3}}V(x)\Bigl[u_{\va}^{(2)}+\theta\big(u_{\va}^{(1)}-u_{\va}^{(2)}\big)\Bigr]^{p_{\va}-3} u_{\va}^{(1)} u_{\va}^{(2)} \xi_\va  \dx
    +O\Big(\big(\lambda_{\va}^{(1)}\big)^{-\frac14} \Big).
  \end{split}
\end{equation}
Then
\begin{equation*}
  \begin{split}
    &\frac{a}{\la_\va^{(1)}}\frac{\lambda_{\va}^{(2)}-\lambda_{\va}^{(1)}}{\|u_{\va}^{(1)}-u_{\va}^{(2)}\|_{L^{\infty}(\mathbb{R}^3)}}\\
    ={}& \Big(\frac{\la_{\va}^{(2)}}{\la_{\va}^{(1)}}-1\Big) \int_{\mathbb{R}^{3}} u_{\va}^{(1)} \xi_{\va} \dx
     -  \frac{p_{\va}-2}{\lambda_{\va}^{(1)}} \int_{\mathbb{R}^{3}}V(x)\Bigl[u_{\va}^{(2)}+\theta\big(u_{\va}^{(1)}-u_{\va}^{(2)}\big)\Bigr]^{p_{\va}-3} u_{\va}^{(1)} u_{\va}^{(2)} \xi_{\va} \dx +O\Big(\big(\lambda_{\va}^{(1)}\big)^{-\frac14} \Big) \\
    ={}& - \frac{p_{\va}-2}{\lambda_{\va}^{(1)}} \int_{\mathbb{R}^{3}} V(x)\Bigl[ u_{\va}^{(2)}+\theta\big(u_{\va}^{(1)}-u_{\va}^{(2)}\big)\Bigr]^{p_{\va}-3} u_{\va}^{(1)} u_{\va}^{(2)} \xi_{\va} \dx 
    +O\Big(\big(\lambda_{\va}^{(1)}\big)^{-\frac14} \Big) .
  \end{split}
\end{equation*}
By a change of variable and \eqref{sum-u}, we have
\begin{eqnarray*}
   &&\frac{a}{\la_\va^{(1)}}\frac{\lambda_{\va}^{(2)}-\lambda_{\va}^{(1)}}{\|u_{\va}^{(1)}-u_{\va}^{(2)}\|_{L^{\infty}(\mathbb{R}^3)}}\\
   &= & - \frac{p_{\va}-2}{\lambda_{\va}^{(1)}} \int_{B_{d}(x^{(1)}_{\va})} V(x) \Bigl[ u_{\va}^{(2)}+\theta\big(u_{\va}^{(1)}-u_{\va}^{(2)}\big)\Bigr]^{p_{\va}-3} u_{\va}^{(1)} u_{\va}^{(2)} \xi_{\va} \dx +O\Big(\big(\lambda_{\va}^{(1)}\big)^{-\frac54} \Big)  \\
   &= &- \frac{p_{\va}-2}{\lambda_{\va}^{(1)}}  \big( \lambda_{\va}^{(1)} \big)^{\frac{p_{\va}-1}{p_{\va}-2} -\frac{3}{2}} V_0^{-\frac{1}{p_{\va}-2} } \int_{\mathbb{R}^{3}} Q_{p_{\va}}^{p_{\va}-1} \bar{\xi}_{\va} \dx \cdot \biggl\{1 + O\bigg(\frac{\ln \la_\va^{(1)}}{\big(\la_\va^{(1)}\big)^{\frac{7}{2}-\frac{4}{p_\va-2}}}\bigg)  \biggr\} 
   +O\Big(\big(\lambda_{\va}^{(1)}\big)^{-\frac54} \Big) \\
 & =& - (p_{\va}-2) V_0^{-\frac{1}{p_{\va}-2} } \big(\lambda_{\va}^{(1)}\big)^{\frac{1}{p_{\va}-2}-\frac{3}{2}}  \int_{\mathbb{R}^{3}} Q_{p_{\va}}^{p_{\va}-1} \bar{\xi}_{\va} \dx 
 +O\bigg(\frac{\ln \la_\va^{(1)}}{\big(\la_\va^{(1)}\big)^{\frac54}}\bigg). 
\end{eqnarray*}
Thus, using \eqref{1-11-6}, we find that for $x \in B_{\sqrt{\lambda_{\va}^{(1)}}d }(0)$,
\begin{eqnarray*}
&&\frac{a}{\la_{\va}^{(1)}} g_{\va}\Big(\frac{x}{\sqrt{\la_{\va}^{(1)}}}+x_{\va}^{(1)}\Big)\\
&=& \frac{a}{\la_\va^{(1)}}\frac{\lambda_{\va}^{(2)}-\lambda_{\va}^{(1)}}{\|u_{\va}^{(1)}-u_{\va}^{(2)}\|_{L^{\infty}(\mathbb{R}^3)}} u_{\va}^{(2)}\Big(\frac{x}{\sqrt{\la_{\va}^{(1)}}}+x_{\va}^{(1)}\Big) \\
&&- \frac{a}{\la_\va^{(1)}} \bigg(\Big|\frac{x}{\sqrt{\la_{\va}^{(1)}}}+x_{\va}^{(1)}\Big|^{-1}\ast \Big[\big(u_{\va}^{(1)}+u_{\va}^{(2)}\big)\xi_{\varepsilon}\Big]\bigg) u_{\va}^{(1)}\Big(\frac{x}{\sqrt{\la_{\va}^{(1)}}}+x_{\va}^{(1)}\Big)  \\
&=& \frac{a}{\la_\va^{(1)}}\frac{\lambda_{\va}^{(2)}-\lambda_{\va}^{(1)}}{\|u_{\va}^{(1)}-u_{\va}^{(2)}\|_{L^{\infty}(\mathbb{R}^3)}} u_{\va}^{(2)}\Big(\frac{x}{\sqrt{\la_{\va}^{(1)}}}+x_{\va}^{(1)}\Big) 
+  O\bigg(\frac{1}{\sqrt{\la_\va^{(1)}}}Q_{p_{\va}} \bigg) \\
&=& - (p_{\va}-2) V_0^{-\frac{2}{p_{\va}-2} } \big(\lambda_{\va}^{(1)}\big)^{\frac{2}{p_{\va}-2}-\frac{3}{2}}  \int_{\mathbb{R}^{3}} Q_{p_{\va}}^{p_{\va}-1} \bar{\xi}_{\va} \dx \cdot \biggl\{Q_{p_{\va}}(x)  
+O\bigg(\frac{\ln \la_\va^{(1)}}{\sqrt{\la_\va^{(1)}}}Q_{p_{\va}} \bigg) \\
 && + O\biggl( \frac{1}{(\lambda_{\va}^{(1)})^{\frac{1}{p_{\va}-2}}} \omega_{\va}^{(2)}\Big(\frac{x}{\sqrt{\la_{\va}^{(1)}}}+x_{\va}^{(1)}\Big) \biggr)\biggr\}
+O\bigg(\frac{\ln \la_\va^{(1)}}{\sqrt{\la_\va^{(1)}}}Q_{p_{\va}} \bigg) .
\end{eqnarray*}
Using the fact that
\[
\begin{aligned}
\big(\la_{\va}^{(1)}\big)^{\frac{2}{p_{\va}-2} -\frac{3}{2}}
={}& V_0^{\frac{2}{p_{\va}-2} } \frac{a}{a_{*, \va}}\bigg(1+O\Big(\frac{\ln \Lambda_\va}{\Lambda_\va^{\frac{7}{2}-\frac{4}{p_\va-2}}}\Big) \bigg)^{\frac{2}{p_{\va}-2} -\frac{3}{2}}
=V_0^{\frac{2}{p_{\va}-2} } \frac{a}{a_{*, \va}}\bigg(1+O\Big(\frac{1}{\Lambda_\va^{\frac{7}{2}-\frac{4}{p_\va-2}}}\Big)\bigg),
\end{aligned}
\]
we obtain that
\[
\begin{split}
&\frac{1}{\la_{\va}^{(1)}} g_{\va}\Big(\frac{x}{\sqrt{\la_{\va}^{(1)}}}+x_{\va}^{(1)}\Big)= -\frac{p_\va-2}{a_{*,\va}}Q_{p_\va}(x)  \int_{\mathbb{R}^{3}} Q_{p_{\va}}^{p_{\va}-1 } \bar{\xi}_{\va} \dx  +
O\biggl(\frac{\ln \la_\va^{(1)}}{\sqrt{\la_\va^{(1)}}}Q_{p_{\va}}(x)
\biggr).
\end{split}
\]
\end{proof}

 The next lemma gives the estimates of  $\xi_{\va}$ in $\r^{3}\setminus B_{R/\sqrt{\la_\va^{(1)}}}(x_{\va}^{(1)})$.
\begin{lemma}\label{lemma5-1}
  There exist constants $C>0$ and $\tau>0$ such that
\begin{equation}\label{5-4-1}
|\xi_{\va}| \leq C e^{-\tau \sqrt{\la_\va^{(1)}}|x- x_{\va}^{(1)}|},  \text{ for any } x \in \r^{3}\setminus  B_{R/\sqrt{\la_{\va}^{(1)}}}(x_{\va}^{(1)}) ,
\end{equation}
and
\begin{equation}\label{5-4-2}
|\nabla \xi_{\va}| \leq C e^{-\tau \sqrt{\la_\va^{(1)}}},  \text{ for any } x \in  \partial B_{d}(x_{\va}^{(1)}) .
\end{equation}
\end{lemma}

\begin{proof}
For large fixed $R>0$, \eqref{5-4-5} and \eqref{5-4-6} imply that
\[
\frac{1}{\la_\va^{(1)}}C_\va(x) \geq \frac{1}{2}, \frac{1}{\la_\va^{(1)}}|g_\va(x)| \leq C  e^{-\sqrt{\la_{\va}^{(1)}}|x-x_{\va}^{(1)}|}, \quad x\in \mathbb{R}^{3}\setminus  B_{R/\sqrt{\la_{\va}^{(1)}}}(x_{\va}^{(1)}).
\]
Using the comparison principle as the proof of Lemma \ref{3-4}, we can prove \eqref{5-4-1} and \eqref{5-4-2}.
\end{proof}

 We now estimate $\xi_{\va}$ in $B_{R/\sqrt{\la_\va^{(1)}}}(x_{\va}^{(1)})$.
Denote $\tilde L_\va$ as follows
  \begin{equation}\label{20-23-5}
   \tilde L_\va \tilde{\xi}_{\va} =- \Delta\tilde{\xi}_{\va} + \Bigl[1-(p_\va-1) Q_{p_\va}^{p_\va-2}(x)\Bigr] \tilde{\xi}_{\va}   + \frac{p_\va-2}{a_{*,\va}}Q_{p_\va}(x)  \int_{\mathbb{R}^{3}} Q_{p_{\va}}^{p_{\va}-1} \tilde{\xi}_{\va} \dx.
  \end{equation}
Following the proof of \cite[Lemma A.1]{Yan2021}, we have
 \begin{lemma}\label{limiteq}
If $ \tilde L_\va \tilde{\xi}_{\va} =0$, then we have
\[
\tilde{\xi}_{\va}(x) = \sum_{j=0}^{3}\ga_{\va,j} \psi_{\va, j},
\]
where $\ga_{\va,j}$ are some constants and
\[
\psi_{\va, 0}= Q_{p_\va} + \frac{p_{\va}-2}{2}x \cdot \nabla Q_{p_\va}, \quad \psi_{\va, j} = \frac{\partial Q_{p_\va}}{\partial x_j}, \text{ for } j=1,2,3.
\]
\end{lemma}

Recall that $\bar{\xi}_{\va}$ is defined in \eqref{0-23-5}.
Thus, we write
\begin{equation}\label{5-4-7}
  \bar{\xi}_{\va}(y) =  \sum_{j=0}^{3}\ga_{\va,j} \psi_{\va, j}(x) + \xi^*_{\va}(x),
\end{equation}
where $\xi^*_{\va}(x)\in \tilde{E}$ with
\begin{equation}\label{6-e}
	\tilde{E} = \{u\in H^1(\mathbb{R}^{3}) : \langle u,\psi_{\va, j}\rangle = 0, \text{ for }  j=0,1,2,3\}.
\end{equation}
Let $\ga_{\va,j}$ be as in \eqref{5-4-7}, we find
\begin{equation}\label{ga-pre}
  \ga_{\va,j} = \frac{\langle \bar{\xi}_{\va}, \psi_{\va, j} \rangle}{\|\psi_{\va, j}\|_{H^1(\mathbb R^3)}^2} = O(\|\bar{\xi}_{\va}\|_{H^1(\mathbb R^3)}) = O(1), j=0,1,2,3.
\end{equation}

From now on, we are going to prove $\ga_{\va,j}$ and $\xi^*_{\va}(x)$ are $o_\va(1)$.
We have the following estimate on $\xi^*_{\va}$.

\begin{proposition}\label{prop-xi}
  Let $\xi^*_{\va}(x)$ be as in \eqref{5-4-7}, then
  \begin{equation}\label{xi}
    \|\xi^*_{\va}\|_{H^1(\mathbb R^3)} =  O\biggl(\frac{\ln \la_{\va}^{(1)}}{\sqrt{\la_{\va}^{(1)}}}\biggr).
  \end{equation}
\end{proposition}

\begin{proof}
  From \eqref{5-4-5}, \eqref{5-4-6} and \eqref{5-4-7}, we have
  \[
  \begin{aligned}
  \tilde{L}_\va(\xi^*_{\va}) ={}& \tilde{L}_\va(\bar{\xi}_{\va}) = -\Delta \bar{\xi}_{\va} + \Big[1-(p_\va-1) Q_{p_\va}^{p_\va-2}\Big] \bar{\xi}_{\va}   +\frac{p_\va-2}{a_{*,\va}}Q_{p_\va}(x)  \int_{\mathbb{R}^{3}} Q_{p_{\va}}^{p_{\va}-1} \bar{\xi}_{\va} \dx\\
  ={}& O\biggl(\frac{\ln \la_{\va}^{(1)}}{\sqrt{\la_{\va}^{(1)}}}\biggr)\bar \xi_{\va} 
   +O\biggl(\frac{\ln \la_\va^{(1)}}{\sqrt{\la_\va^{(1)}}}Q_{p_{\va}}(x)\biggr).
  \end{aligned}
  \]
  By Lemma~\ref{limiteq}, in view of $\xi^*_{\va}\in \tilde{E} $,  it is standard to prove
  \begin{equation}\label{1-13-6}
   \|\xi^*_{\va}\|_{H^1(\mathbb R^3)}
   \le C\|\tilde{L}_\va(\xi^*_{\va})\|_{L^2(\mathbb R^3)} 
   = O\biggl(\frac{\ln \la_{\va}^{(1)}}{\sqrt{\la_{\va}^{(1)}}}\biggr).
   \end{equation}
\end{proof}

\begin{lemma} It holds
  \begin{equation}\label{ga0}
    \ga_{\va,0} = o_\va(1) .
  \end{equation}
\end{lemma}

\begin{proof}
 By \eqref{5-4-7}, Lemma \ref{limiteq} and Proposition \ref{prop-xi}, we have
  \begin{eqnarray}\label{ga0-3-1}
  &&  \int_{B_{d}(x_{\va}^{(1)})}\big(u_{\va}^{(1)}+u_{\va}^{(2)}\big)\xi_{\va} \dx \nonumber\\
&  =& 2V_0^{-\frac{1}{p_\va-2}} \big(\la_{\va}^{(1)}\big)^{ \frac{1}{p_{\va}-2} -\frac{3}{2}} \int_{B_{d\sqrt{\la_{\va}^{(1) } }}(0)} Q_{p_{\va}} \bar{\xi}_{\va} \dx  + O\Bigl( \big(\la_{\va}^{(1)}\big)^{ \frac{1}{p_{\va}-2} -2} \ln \la_{\va}^{(1)} \Bigr) \nonumber\\
  & =& 2V_0^{-\frac{1}{p_\va-2}} \big(\la_{\va}^{(1)}\big)^{ \frac{1}{p_{\va}-2} -\frac{3}{2}} \int_{B_{d\sqrt{\la_{\va}^{(1) } }}(0)} Q_{p_{\va}} \bigg(\sum_{j=0}^{3}\ga_{\va,j}\psi_{\va, j} + \xi^*_{\va}\bigg) \dx 
 + O\Bigl( \big(\la_{\va}^{(1)}\big)^{ \frac{1}{p_{\va}-2} -2}\ln \la_{\va}^{(1)} \Bigr)\nonumber \\
&  =& 2V_0^{-\frac{1}{p_\va-2}} \big(\la_{\va}^{(1)}\big)^{ \frac{1}{p_{\va}-2} -\frac{3}{2}} \ga_{\va, 0} \int_{B_{d\sqrt{\la_{\va}^{(1) } }}(0)} Q_{p_{\va}}\psi_{\va, 0} \dx  +  O\Bigl( \big(\la_{\va}^{(1)}\big)^{ \frac{1}{p_{\va}-2} -2}\ln \la_{\va}^{(1)} \Bigr) \nonumber \\
  &=& 2V_0^{-\frac{1}{p_\va-2}} \big(\la_{\va}^{(1)}\big)^{ \frac{1}{p_{\va}-2} -\frac{3}{2}}  \ga_{\va, 0} \int_{\r^{3}} Q_{p_{\va}} \big(Q_{p_{\va}} + \frac{p_{\va}-2}{2} x \cdot \nabla Q_{p_{\va}}\big) \dx  + O\Bigl( \big(\la_{\va}^{(1)}\big)^{ \frac{1}{p_{\va}-2} -2}\ln \la_{\va}^{(1)} \Bigr) \nonumber \\
 & =& 2V_0^{-\frac{1}{p_\va-2}}\big(\la_{\va}^{(1)}\big)^{ \frac{1}{p_{\va}-2} -\frac{3}{2}} \bigg(1 -\frac{3(p_{\va}-2)}{4}\bigg) \ga_{\va,0}  \int_{\r^{3}} Q^{2}_{p_{\va}} \dx 
 +O\Bigl( \big(\la_{\va}^{(1)}\big)^{ \frac{1}{p_{\va}-2} -2}\ln \la_{\va}^{(1)} \Bigr) \nonumber \\
&  =& \frac 32\Big(\frac{10}{3} -p_{\va} \Big) V_0^{-\frac{1}{p_\va-2}} \ga_{\va, 0} \big(\la_{\va}^{(1)}\big)^{ \frac{1}{p_{\va}-2} -\frac{3}{2}} a_{*, \va}   
+O\Bigl( \big(\la_{\va}^{(1)}\big)^{ \frac{1}{p_{\va}-2} -2}\ln \la_{\va}^{(1)} \Bigr). \nonumber 
  \end{eqnarray}
Thus, we obtain
  \begin{eqnarray}\label{ga0-2} \int_{B_{d}(x_{\va}^{(1)})}\big(u_{\va}^{(1)}+u_{\va}^{(2)}\big)\xi_{\va} \dx  =
      \frac 32\Big(\frac{10}{3} -p_{\va} \Big) V_0^{-\frac{1}{p_\va-2}} \big(\la_{\va}^{(1)}\big)^{ \frac{1}{p_{\va}-2} -\frac{3}{2}} a_{*, \va} \ga_{\va,0}  
      +O\Bigl( \big(\la_{\va}^{(1)}\big)^{ \frac{1}{p_{\va}-2} -2}\ln \la_{\va}^{(1)} \Bigr).
  \end{eqnarray}
Moreover, using
\[
\int_{\mathbb R^3} \big(u_{\va}^{(1)}+u_{\va}^{(2)}\big)\xi_{\va} \dx  =\frac1{\|u_{\va}^{(1)}-u_{\va}^{(2)}\|_{L^\infty(\mathbb R^3)}}
\int_{\mathbb R^3}\Big[ \big(u_{\va}^{(1)}\big)^2-\big(u_{\va}^{(2)}\big)^2\Big] \dx 
=0,
\]
and Lemma~\ref{3-1}, we find
\begin{eqnarray}\label{ga0-1} \int_{B_{d}(x_{\va}^{(1)})}\big(u_{\va}^{(1)}+u_{\va}^{(2)}\big)\xi_{\va} \dx =
     \int_{\mathbb R^3 \setminus B_{d}(x_{\va}^{(1)})}\big(u_{\va}^{(1)}+u_{\va}^{(2)}\big)\xi_{\va} \dx =
     O\Bigl(e^{-\theta \sqrt{ \la_{\va}^{(1)} }}  \Bigr).
\end{eqnarray}
Hence, from  \eqref{ga0-2} and \eqref{ga0-1}, in view of
\[
\frac{10}{3} -p_{\va}= \pm \va = O\Bigl( \big(\ln \la_{\va}^{(1)}\big)^{-1} \Bigr),
\]
we get
$\ga_{\va, 0}=o_\va(1).$
  \end{proof}

\begin{lemma} It holds
  \begin{equation}\label{ga12}
    \ga_{\va,j} = o_\va(1), \, j=1,2, 3.
  \end{equation}
\end{lemma}

\begin{proof}
  Using \eqref{Pohozaev0}, we have the following identity
  \begin{equation}\label{p61-23-1}\begin{split}
  		&\frac{1}{p_{\va} } \int_{ B_{d}(x^{(1)}_{\va})} \frac{\partial V(x)}{\partial x_j} F\big(p_{\va}, u_{\va}^{(1)}, u_{\va}^{(2)}\big)  \xi_{\va}   \dx \\
  		=  {}& \int_{\partial B_{d}(x^{(1)}_{\va})} \frac{\partial  u_{\va}^{(2)}}{\partial \nu} \frac{\partial \xi_{\va}}{\partial x_{j}} \dsi
  		+ \int_{\partial B_{d}(x^{(1)}_{\va})} \frac{\partial  \xi_{\va}}{\partial \nu} \frac{\partial u_{\va}^{(1)}}{\partial x_{j}} \dsi
  		- \frac{1}{2}  \int_{\partial B_{d}(x^{(1)}_{\va})} \nabla \big(u_{\va}^{(1)} + u_{\va}^{(2)}\big) \nabla \xi_{\va} \nu_{j}  \dsi \\
  		{}&  -\frac{\la_\va^{(1)}-\la_\va^{(2)}}{2\|u_{\va}^{(1)}-u_{\va}^{(2)}\|_{L^3(\r^3)}} \int_{\partial B_{d}(x^{(1)}_{\va})} \big(u_{\va}^{(1)}\big)^2 \nu_j \dsi
  		- \frac{\la_\va^{(2)}}{2} \int_{\partial B_{d}(x^{(1)}_{\va})}  \big(u_{\va}^{(1)} + u_{\va}^{(2)}\big) \xi_{\va}  \nu_{j} \dsi  \\
  		{}&+ \frac{1}{p_{\va} } \int_{\partial B_{d}(x^{(1)}_{\va})} V(x) F\big(p_{\va}, u_{\va}^{(1)}, u_{\va}^{(2)}\big)  \xi_{\va}  \nu_{j} \dsi \\
  		{}&-\frac12 \int_{\partial B_{d}(x^{(1)}_{\va})} \int_{\r^3} \frac{\big(u_{\va}^{(1)}(y)+u_{\va}^{(2)}(y) \big) \xi_{\la}(y)}{|x-y|} \dy \cdot  \big(u_{\va}^{(1)}(x) \big)^2 \nu_j \dsi \\
  		{}&-\frac12 \int_{\partial B_{d}(x^{(1)}_{\va})} \int_{\r^3} \frac{\big(u_{\va}^{(2)}(y)\big)^2 }{|x-y|} \dy \cdot  \big(u_{\va}^{(1)}(x)+u_{\va}^{(2)}(x) \big) \xi_{\va}(x) \nu_j \dsi, \\
  		{}&+\frac12 \int_{\r^3\setminus B_{d}(x^{(1)}_{\va})} \int_{\r^3} \frac{x_j-y_j}{|x-y|^3}\big(u_{\va}^{(1)}(y)+u_{\va}^{(2)}(y) \big) \xi_{\va}(y) \dy \cdot  \big(u_{\va}^{(1)}(x) \big)^2  \dx \\
  		{}&+\frac12 \int_{\r^3\setminus B_{d}(x^{(1)}_{\va})} \int_{\r^3} \frac{x_j-y_j}{|x-y|^3} \big(u_{\va}^{(2)}(y)\big)^2\dy \cdot  \big(u_{\va}^{(1)}(x)+u_{\va}^{(2)}(x) \big) \xi_{\va}(x)  \dx \\
  	\end{split}
  \end{equation}
  where
  \[ F\big(p_{\va}, u_{\va}^{(1)}, u_{\va}^{(2)}\big)= \frac{\big(u_{\va}^{(1)}\big)^{p_{\va}} - \big(u_{\va}^{(2)}\big)^{p_{\va}} }{u_{\va}^{(1)} - u^{(2)}_{\va}} = p_{\va} \int_{0}^{1}\big(tu_{\va}^{(1)} +(1-t)u_{\va}^{(2)}\big)^{p_{\va}-1}\dt.  \]
  .

  By the exponential decay of $u_{\va}^{(1)}, u_{\va}^{(1)}$ and $\xi_{\va}$,
   we can deduce from \eqref{p61-23-1} that
  \begin{equation}\label{61-23-5}
  \begin{split}
    \frac{1}{p_{\va} } \int_{ B_{d}(x^{(1)}_{\va})} \frac{\partial V(x)}{\partial x_j} F\big(p_{\va}, u_{\va}^{(1)}, u_{\va}^{(2)}\big)  \xi_{\va}   \dx
    = O\Big(e^{-\tau \sqrt{\la_{\va}^{(1)}}}\Big).
  \end{split}
  \end{equation}
  By a change of variable, we have
  \begin{equation}\label{gai-1}
  \begin{split}
   \int_{B_{d \sqrt{\la_{\va}^{(1)}}}(0)} \frac{\partial V\Big(\frac{x}{\sqrt{\la_{\va}^{(1)}}} + x_{\va}^{(1)}\Big)}{\partial x_{j}} F\bigg(p_{\va}, u_{\va}^{(1)}\Big(\frac{x}{\sqrt{\la_{\va}^{(1)}}} + x_{\va}^{(1)}\Big), u_{\va}^{(2)}\Big(\frac{x}{\sqrt{\la_{\va}^{(1)}}} + x_{\va}^{(1)}\Big)\bigg)  \bar{\xi}_{\va} \dx 
    = O\Big(e^{-\tau\sqrt{\la_{\va}^{(1)}}}\Big).
  \end{split}
  \end{equation}
  It follows from the fact $|x_{\va}^{(1)}-b_0| = O\Big(\big(\la_{\va}^{(1)}\big)^{-1}\Big)$ that
  \[
  \frac{\partial V(x_{\va}^{(1)})}{\partial x_{j}} = \frac{\partial V(b_0)}{\partial x_{j}} + O\big(|x_{\va}^{(1)}-b_0|\big) = O\Big(\big(\la_{\va}^{(1)}\big)^{-1}\Big),
  \]
  which implies
  \begin{equation}\label{gai-2}
  \int_{B_{\sqrt{\la_{\va}^{(1)}}d}(0)} \frac{\partial V(x_{\va}^{(1)})}{\partial x_{j}} F\bigg(p_{\va}, u_{\va}^{(1)}\Big(\frac{x}{\sqrt{\la_{\va}^{(1)}}} + x_{\va}^{(1)}\Big), u_{\va}^{(2)}\Big(\frac{x}{\sqrt{\la_{\va}^{(1)}}} + x_{\va}^{(1)}\Big)\bigg)  \bar{\xi}_{\va} \dx 
    = O\Big(\big(\la_{\va}^{(1)}\big)^{\frac{1}{p_{\va}-2}}\Big).
  \end{equation}
  Hence, it follows from \eqref{gai-1} and \eqref{gai-2} that
  \[
  \sum_{h=1}^{3} \int_{B_{\sqrt{\la_{\va}^{(1)}}d}(0)} \frac{\partial^2 V(x_{\va}^{(1)})}{\partial x_{j} \partial x_{h}} x_{h} F\bigg(p_{\va}, u_{\va}^{(1)}\Big(\frac{x}{\sqrt{\la_{\va}^{(1)}}} + x_{\va}^{(1)}\Big), u_{\va}^{(2)}\Big(\frac{x}{\sqrt{\la_{\va}^{(1)}}} + x_{\va}^{(1)}\Big)\bigg)  \bar{\xi}_{\va} \dx 
    = O\Big(\big(\la_{\va}^{(1)}\big)^{\frac{1}{p_{\va}-2}}\Big).
  \]
  By \eqref{sum-u}, we have
  \[
  \begin{aligned}
  &F\bigg(p_{\va}, u_{\va}^{(1)}\Big(\frac{x}{\sqrt{\la_{\va}^{(1)}}} + x_{\va}^{(1)}\Big), u_{\va}^{(2)}\Big(\frac{x}{\sqrt{\la_{\va}^{(1)}}} + x_{\va}^{(1)}\Big)\bigg) \\
  ={}&p_\va \Big(\frac{\la_{\va}^{(1)}}{V_0}\Big)^{\frac{1}{p_{\va}-2}+1} Q_{p_\va}^{p_\va-1}(x) \bigg[1 + O\bigg(\frac{\ln \la_\va^{(1)}}{(\la_\va^{(1)})^{\frac{7}{2}-\frac{4}{p_\va-2}}}\bigg)  \bigg] \\
  ={}&p_\va \Big(\frac{\la_{\va}^{(1)}}{V_0}\Big)^{\frac{1}{p_{\va}-2}+1} Q_{p_\va}^{p_\va-1}(x) + O\Big( \big(\la_{\va}^{(1)}\big)^{\frac{1}{p_{\va}-2}+\frac12} \big(\ln \la_\va^{(1)}\big) Q_{p_\va}^{p_\va-1}(x) \Big) .
  \end{aligned}
  \]
  Then
  \[\begin{split}
  & \sum_{h=1}^{3} \int_{B_{\sqrt{\la_{\va}^{(1)}}d}(0)} \frac{\partial^2 V(x_{\va}^{(1)})}{\partial x_{j} \partial x_{h}} x_{h} Q_{p_\va}^{p_\va-1}(x) \sum_{k=0}^{3}\ga_{\va,k} \psi_{\va, k}(x) \dx \\
  ={}&  \sum_{h=1}^{3} \ga_{\va,h} \frac{\partial^2 V(x_{\va}^{(1)})}{\partial x_{j} \partial x_{h}} \int_{\mathbb{R}^{3}} x_{h} Q_{p_\va}^{p_\va-1} \frac{\partial Q_{p_\va}}{\partial x_{h}} \dx\\
    ={}&  O\Bigl(\big(\la_{\va}^{(1)}\big)^{-\frac{1}{2}} \ln \la_\va^{(1)} \Bigr). 
 \end{split} \]
  Since
  \[
   \int_{\mathbb{R}^{3}} x_{h} Q_{p_\va}^{p_\va-1} \frac{\partial Q_{p_\va}}{\partial x_{h}} \dx
  = -\frac{1}{p_\va} \int_{\mathbb{R}^{3}} Q_{p_\va}^{p_\va} \dx  < 0,
  \]
and  $b_0$ is a non-degenerate critical point of $V(x)$, we conclude that
   \[\ga_{\va,j} =O\Bigl(\big(\la_{\va}^{(1)}\big)^{-\frac{1}{2}}\ln \la_\va^{(1)}  \Bigr), \, j=1,2, 3.\]
\end{proof}

\begin{proof}[\bf Proof of Theorem \ref{UniquenessOfNormalized}.]
	 On one hand, Lemma \ref{lemma5-1} shows that
	\[
	\xi_{\va}(x) = o_\va(1), \quad x\in \mathbb{R}^{3}\setminus  B_{R/\sqrt{\la_{\va}^{(1)}}}(x_{\va}^{(1)}).
	\]
	On the other hand, it follows from \eqref{xi}, \eqref{ga0} and \eqref{ga12} that
	\[
	\xi_{\va}(x) = o_\va(1), \quad x\in  B_{R/\sqrt{\la_{\va}^{(1)}}}(x_{\va}^{(1)}).
	\]
	Then, $\xi_{\va}(x) = o_\va(1), \, x\in\mathbb{R}^{3}$. It contradicts to $\|\xi_{\va}\|_{L^{\infty}(\mathbb{R}^3)}=1$. Thus, we have $u_{\va}^{(1)}(x)\equiv u_{\va}^{(2)}(x)$ for small $\va > 0$.
\end{proof}

\appendix

\section{A useful estimate and Hardy-Littlewood-Sobolev inequality}

\renewcommand{\theequation}{A.\arabic{equation}}

\begin{lemma}\label{4-1}(\cite[ Lemma A.1]{Guo}) For any fixed small constant $\theta>0$, there exists
	$C_1>0$, such that
	\[
	Q_{p_\va}(x) \leq C_1 e^{-(1-\theta)|x|}, \quad \forall x\in\r^3.
	\]
	Moreover,
	\[
	\| Q_{p_{\va}}- Q_{\bar{p}}\|_{H^1(\mathbb R^3)} \leq C \va,
	\]
	for some $C>0$, independent of $\va$.
\end{lemma}

\begin{lemma}{(Hardy-Littlewood-Sobolev inequality, see \cite{Lieb})}\label{lem-HLS}
	Let $p,r>1$ and $0<t< N$ with $1/p+t/N+1/r=2$. Let $f\in L^{p}(\r^N)$ and $h\in L^{r}(\r^N)$. Then there exists a sharp constant $C(N,t,p)$, independent of $f$ and $h$, such that
	\begin{equation}\label{HLS}
		\left| \int_{\r^N}\int_{\r^N} f(x) |x-y|^{-t} h(y) \dx\dy  \right| \leq C(N,t,p) \|f\|_{L^{p}(\r^N)} \|h\|_{L^{r}(\r^N)}.
	\end{equation}
\end{lemma}

\medskip

\noindent
\textbf{Data Availability Statement.} The authors confirm that there is no data used in our manuscript.

\noindent
\textbf{Acknowledgment.} The authors would like to thank Professor Peng Luo from Central China Normal University for the helpful discussion with him. This paper was supported by NSFC grants (No.12071169).   

\end{document}